\documentclass[reqno]{amsart}
\usepackage{hyperref}
\usepackage[a4paper, left=3cm, right=3cm, top=4cm, bottom=4cm]{geometry}
\vspace{9mm}

\begin{document}

\title[\hfilneg \hfil Controllability]{A Study of nonlocal fractional neutral stochastic integrodifferential inclusions of order $1<\alpha<2$ with impulses}
\author[A. Afreen, A. Raheem \&  A. Khatoon \hfil \hfilneg]
{A. Afreen$^{*}$, A. Raheem \& A. Khatoon}

\address{A. Afreen \newline
	Department  of Mathematics,
	Aligarh Muslim University,\newline Aligarh -
	202002, India.} \email{afreen.asma52@gmail.com}

\address{A. Raheem \newline Department of Mathematics,
	Aligarh Muslim University,\newline Aligarh -
	202002, India.} \email{araheem.iitk3239@gmail.com}

\address{A. Khatoon \newline
	Department  of Mathematics,
	Aligarh Muslim University,\newline Aligarh -
	202002, India.} \email{areefakhatoon@gmail.com}

\renewcommand{\thefootnote}{} \footnote{$^*$ Corresponding author:
	\url{ A. Afreen (afreen.asma52@gmail.com)}}



\begin{abstract} This paper considers a class of nonlocal fractional neutral stochastic integrodifferential inclusions of order $1<\alpha<2$ with impulses in a Hilbert space. We study the existence of the mild solution for the cases when
	the multi-valued map has convex and non-convex values. The results are obtained by combining fixed-point theorems with the fractional order cosine family, semigroup theory, and stochastic techniques. A new set of sufficient conditions is developed to demonstrate the approximate controllability of the system. Finally, an example is
	given to illustrate the obtained results.
\end{abstract}

\maketitle \numberwithin{equation}{section}
\newtheorem{theorem}{Theorem}[section]
\newtheorem{lemma}[theorem]{Lemma}
\newtheorem{proposition}[theorem]{Proposition}
\newtheorem{corollary}[theorem]{Corollary}
\newtheorem{remark}[theorem]{Remark}
\newtheorem{definition}[theorem]{Definition}
\newtheorem{example}[theorem]{Example}
\allowdisplaybreaks

\section{\textbf{Introduction}}
 The limitations of using traditional integer-order calculus and standard differential equations to model specific complex systems and processes are inadequate for describing various phenomena, including those found in viscoelastic systems, dielectric polarization, and electromagnetic wave propagation. This insufficiency is sometimes caused by these events exhibiting memory-like behavior or other complex features that conventional derivatives cannot adequately represent.
 Many researchers have turned to fractional derivatives as an alternative to these limitations. Fractional derivatives offer a way to extend the concept of differentiation beyond integer orders. They are a powerful approach for modeling processes with memory and complex behaviors in diverse scientific and engineering fields.
 
 In recent decades, the theory of fractional-order differential systems has attracted considerable attention mainly due to its impressive applicability in various scientific and engineering domains, such as medical model simulations and electrical engineering. The nonlocal properties of fractional differential and integral operators have made them valuable tools. It has been found that arbitrary-order fractional differential equations provide a more accurate description of the dynamic response of real-world objects. The importance of a thorough understanding of fractional calculus has increasingly been recognized \cite{a11,i1,k3}.
 
 It is well-established that many phenomena in the physical universe, including the human cardiovascular system and blood circulation, are influenced by the abrupt shift in their state after certain intervals. Impulsive differential equations describing sudden changes are used to explain these events. The study of impulsive systems is crucial to analyzing more realistic mathematical models. It has numerous applications in various fields, such as drug diffusion in the human body, population dynamics, theoretical physics, industrial robotics, mathematical economy, chemical technology, etc. Ahmed et al. \cite{HMBR}  examined the approximate controllability of noninstantaneous impulsive Hilfer fractional integrodifferential equations with fractional Brownian motion. Sivasankar and Udhayakumar \cite{SU2022} studied the approximate controllability of second-order
 neutral stochastic delay integrodifferential evolution
 inclusions with impulses. Refer to \cite{MAJVN2022,LBS1989,GSXY} for more details.

 Stochastic differential equations (SDEs) are widely used to model real-world phenomena influenced by random factors or noises. Unlike deterministic models, which describe how a system evolves based on fixed parameters, stochastic models account for the inherent randomness in the system. They incorporate the concept of Brownian motion and stochastic processes to capture the unpredictable nature of the phenomena being studied.
 
 Stochastic analysis, which deals with SDEs, has applications in various disciplines. SDEs are extensively used in financial modeling and option pricing. They provide a framework to account for random market fluctuations and uncertainty. In physics and engineering, SDEs are used in modeling systems that exhibit unexpected behavior, such as heat conduction in materials with memory. They are also applied to describe diffusion processes, particle motion in fluids, and other phenomena influenced by random forces. Stochastic models are valuable for understanding biological systems, where random fluctuations play a significant role. They can help simulate the dynamics of biochemical reactions, population growth, genetic mutations, and epidemiological processes. Chemical reactions often involve random events, such as molecular collisions. Stochastic models are employed to analyze reaction rates, reaction pathways, and the behavior of chemical systems under uncertainty. Stochastic models are essential in economic analysis and forecasting. They are used to model market dynamics, price movements, and the behavior of economic agents under uncertainty \cite{AMAAM,Hamdy,Hamdy2017,E2013,n11}.
 
 One of the fundamental concepts of contemporary control theory is the idea of a dynamical system being controllable. In general, controllability is the capability of a control system to be directed from an arbitrary initial state to a likewise arbitrary final state through the permitted set of controls. Significant consequences for the behavior of linear and nonlinear dynamical systems are drawn from this idea. Moreover,
 the approximate controllability means that the system can be steered to an arbitrarily
 small neighborhood of the final state. Approximate controllable systems are more prevalent, and very often, approximate controllability is completely adequate in applications, see \cite{AAA2022,RAK2022,RK2019,Hamdy2009,Hamdy2022,MVM2016,VPNB2021} and the references therein.  
 

 It should be pointed out that Sathiyaraj and Balasubraminaim \cite{SB2018} discussed the controllability of fractional higher-order stochastic integrodifferential inclusions:
 \begin{eqnarray*}
 	^{C}{D}^{q}x(t) \in A ^{C}{D}^{p}x(t)+B u(t)+F(t,x(t))+\displaystyle\int_{0}^{t}{G}(s,x(s))dw(s), \quad t\in[0,T],
 	\\x(0)=x_0,~ x'(0)=x'_0,\hspace{8cm}
 \end{eqnarray*}
 where $^{C}{D}^{q}$ and $^{C}{D}^{p}$ denote Caputo fractional derivatives of order $0<p\le1<q\le 2,$ $A$ and $B$ are matrices of dimensions $n \times n$ and $n \times m,$ respectively, and $x\in \mathbb{R}^n,~ u\in \mathbb{R}^m$ are the state and control vectors, respectively. The results are established using the fixed point techniques for both the convex and non-convex cases. It is noted that the study is in finite dimension.

 Sivasankar and Udhayakumar \cite{SU2022} studied the existence of solutions and the approximate controllability of the following integrodifferential evolution inclusions with infinite delay
 \begin{eqnarray*} 
 	\left \{ \begin{array}{lll} d\big[x'(t)
 		+\sigma(t,x_t)\big]\in \big[ {A}(t)x(t)+{B}u(t)+F(t,x_t)\big]dt\\
 		\hspace{4cm}+\displaystyle\int_{0}^{t}e(t-s){G}(s,x_s)dw(s),\quad t \in [0,a],~t \neq t_k,~k=1,2,\ldots,m,
 		&\\ x(0)=\phi\in L^2(\Omega, B_h), \quad t\in(-\infty,0],\quad x'(0)= \zeta,
 		&\\ \Delta x\rvert_{t=t_k}={I}_k(x_{t_k}), \quad k=1,2,\ldots,m,  &\\
 		\Delta x'\rvert_{t=t_k}={J}_k(x_{t_k}),\quad k=1,2,\ldots,m,
 	\end{array}\right.
 \end{eqnarray*}
 where $x(\cdot)$ takes values in a Hilbert space. The results are proved only for convex multivalued maps using fixed point methods.

 Several real-world systems, such as sudden stock price fluctuations brought on by war, epidemics, market crashes, etc., are subject to stochastic perturbations with impulsive effects. For these models, the path continuity supposition could be more
 plausible. Consequently, we should consider impulsive
 effects in stochastic processes when modeling such systems. Due to the numerous applications of second-order abstract differential equations in physics and engineering, they have received much more attention. Second-order differential equations, for instance, can be used to analyze the system of dynamical buckling of a hinged extensible beam \cite{S1950,WE1982}. The integrated process in continuous time, which can be made stationary, is best
 described by second-order stochastic differential equations. For instance, engineers
 can use second-order stochastic differential equations to describe mechanical
 vibrations or the charge on a capacitor. Researchers have recently obtained numerous findings about
 the controllability of second-order stochastic control systems \cite{AMAAM,SSM2019,MVU2020,SU2022,VPNB2021}.
 
 Numerous social, physical, biological, and engineering issues are characterized mainly by differential inclusions; for instance, see \cite{MAJVN2022,SU2022}. Ahmed and Ragusa \cite{HM2022} examined the nonlocal controllability of Sobolev-type conformable
 fractional stochastic evolution inclusions with Clarke
 subdifferential. Kavitha et al. \cite{KVSNU2021} studied the approximate controllability of Sobolev-type fractional neutral
 differential inclusions of Clarke subdifferential type. Subsequently, a few authors have reviewed the controllability of fractional
 stochastic differential inclusions involving Hilfer fractional derivatives—for example, Ahmed et al. \cite{HMWA2022} studied the null controllability of Hilfer fractional differential inclusions. Ahmed et al. \cite{HMWA2023} studied fractional stochastic evolution inclusions with control on the boundary. Many authors examined the existence and controllability of various systems with nonlocal conditions as their application results are more favorable than those of the classical initial condition \cite{GNR2007,SSM2019,BGN2004}. Many authors studied the various controllability of Hilfer fractional
 differential systems \cite{Hamdy2014,AW2018,HMAM2019}. To the authors' knowledge, the existence of mild solutions of nonlocal fractional neutral stochastic integrodifferential inclusions of order $1<\alpha<2$ with impulses in Hilbert spaces has not been addressed for the convex and non-convex case of the multivalued map. This work aims to address this gap.

 In light of the above discussion, we consider the following nonlocal fractional neutral stochastic integrodifferential inclusions with impulses
 \begin{eqnarray} \label{1.1}
 	\left \{ \begin{array}{lll} {}^C{D}^{\alpha}\big[\zeta(x)
 		+f_1\big(x,\zeta(x),\zeta(\nu_1(x))\big)\big]\in  \mathcal {A}\zeta(x)+\mathcal{B}u(x)+f_2\big(x,\zeta(x),\zeta(\nu_2(x))\big)\\
 		\hspace{1.4cm}+\displaystyle\int_{0}^{x}\varrho(x-y)\mathcal{G}\big(y,\zeta(y),\zeta(\nu_3(y))\big)dw(y),\quad x \in [0,b],~x \neq x_p,~p=1,2,\ldots,n,
 		&\\ \zeta(x)=\phi(x)+h(\zeta)(x), \quad x\in[-a,0],\quad \zeta'(0)= \xi \in \mathcal{H},
 		&\\ \Delta \zeta(x_p)=\hat{I}_p(\zeta(x_p)), \quad p=1,2,\ldots,n,  &\\
 		\Delta \zeta' (x_p)=\hat{J}_p(\zeta(x_p)),\quad p=1,2,\ldots,n,
 		
 	\end{array}\right.
 \end{eqnarray}
 where $0<a, b<\infty,$ $1<\alpha< 2$, $^{C}{D}^{\alpha}$ represents fractional derivative in the Caputo sense and $\zeta(\cdot)$ takes values in a real separable Hilbert space $\mathcal{H}$ with inner product $(\cdot,\cdot)$ and norm $\|\cdot\|.$ A densely defined closed linear operator $\mathcal{A}:\mathcal{D(A)}\subset \mathcal{H}\rightarrow \mathcal{H}$ generates a strongly continuous cosine family $\{C(x)\}_{x\geq 0}$ in $\mathcal{H}$. $\mathcal{B}$ is bounded linear operator from a real separable Hilbert space $\mathcal{U}$ to $\mathcal{H},$ and $u(\cdot)\in L^2([0,b],\mathcal{U}),$ a Hilbert space of admissible control functions. Let $\mathcal{W}$ be another real separable Hilbert space with inner product $(\cdot,\cdot)_{\mathcal{W}}$ and norm $\|\cdot\|_{\mathcal{W}}.$ Let $w$ be a $\mathcal{W}$-valued Wiener process with a finite trace nuclear covariance operator $Q\geq 0$ on a complete probability space $(\Omega, \Upsilon,\Upsilon_x, P),$ where $\Upsilon_x\subset \Upsilon,$ $x \in [0,b]$ is a normal filtration. $\Upsilon_x$ is a right continuous increasing family and $\Upsilon_0$ contains all $P$-null sets. Also, let $L_2^{0}=L_2\big(Q^{1/2}\mathcal{W},\mathcal{H}\big),$  the space of all Hilbert-Schmidt operators from $Q^{1/2}\mathcal{W}$ to~$\mathcal{H}$ be a separable Hilbert space with the norm $\|\psi\|^2_{Q}=tr[\psi Q \psi^*].$ The maps $f_1, f_2:[0,b] \times \mathcal{H}^2\to \mathcal{H},$ $\mathcal{G}:[0,b] \times \mathcal{H}^2\to 2^{L(\mathcal{W},\mathcal{H})}\setminus \{\emptyset\},$ $\varrho:[0,b]\rightarrow \mathbb{R}$ and
 $\hat{I}_p, \hat{J}_p:\mathcal{H}\rightarrow \mathcal{H}$ are appropriate functions. $0=x_0<x_1<x_2<\cdots<x_n<x_{n+1}=b$ are impulse points and $\Delta \zeta(x_p)=\zeta(x_p^+)-\zeta(x_p^-)$ denotes the jump of a function $\zeta$ at $x_p.$ For $i=1,2,3,$ the functions $\nu_i:[0,b] \rightarrow [-a,b]$ are 
 continuous such that $0 \leq \nu_i(x) \leq x,$ 
 $x \in[0,b]$.

 The set of all strongly measurable, square integrable $\mathcal{H}$-valued random variables, denoted by $L_2(\Omega, \mathcal{H}),$ is a Banach space with norm $\|\zeta\|_{L_2}=\Big({E}\|\zeta(x)\|^2\Big)^{\frac{1}{2}}.$ Let $C([0,b],L_2(\Omega, \mathcal{H}))$ be the Banach space of all continuous maps from $[0,b]$ to $L_2(\Omega, \mathcal{H})$ satisfying $\sup\limits_{x\in[0,b]}{E}\|\zeta(x)\|^2<\infty.$ A significant subspace is given by $L_2^0(\Omega, \mathcal{H})=\{\Upsilon \in L_2(\Omega, \mathcal{H}): \Upsilon ~\mbox{is}~ \Upsilon_0\mbox{-measurable}\}.$

 We define the space	$\mathcal{PC}_{0}=C([-a,0], \mathcal{H}):$ The space of all $\Upsilon_0$-adapted, $\mathcal{H}$-valued measurable stochastic processes $\zeta:[-a,0]\rightarrow \mathcal{H},$ which are continuous everywhere, forms a Banach space with the norm  $\| \zeta \|_{\mathcal{PC}_0}=\sup\limits_{x \in [-a,0]}{E}\Big( \|\zeta(x)\|^2 \Big)^{\frac{1}{2}}.$
 
 \noindent
 $\mathcal{PC}_{b}=PC([-a,b], \mathcal{H}):$ The space of all $\Upsilon_x$-adapted, $\mathcal{H}$-valued measurable stochastic processes $\zeta:[-a,b]\rightarrow \mathcal{H},$ which are continuous everywhere except the points $x_p\in (0,b), ~p=1,2,\ldots,n.$ Moreover, right limits $\zeta(x_p^+)$ and left limits $\zeta(x_p^-)$ of $\zeta$ exist at $x_p,~p=1,2,\ldots,n,$ and $\zeta$ is left continuous, i.e., $\zeta(x_p^-)=\zeta(x_p).$ The space $\mathcal{PC}_{b}$ is a Banach space equipped with the norm $\| \zeta \|_{\mathcal{PC}_b}=\sup\limits_{x \in [-a,b]}{E}\Big( \|\zeta(x)\|^2 \Big)^{\frac{1}{2}}.$ Similarly, $\mathcal{PC}=PC([0,b], \mathcal{H})$ is a Banach space equipped with the norm $\| \zeta \|_{\mathcal{PC}}=\sup\limits_{x \in [0,b]}{E}\Big( \|\zeta(x)\|^2 \Big)^{\frac{1}{2}}.$\\
 \noindent
 Let  $\mathcal{PC}^1_{b}$ be the space of all functions $\zeta \in \mathcal{PC}_{b},$ which are continuously differentiable on $[0,b]-\{x_1,x_2,\ldots , x_p\},$ and the lateral derivatives 
 $$\zeta'_R(x)=\lim\limits_{y\rightarrow 0+}\frac{\zeta(x+y)-\zeta(x^+)}{y} ~~\mbox{and}~~ \zeta'_L(x)=\lim\limits_{y\rightarrow 0-}\frac{\zeta(x+y)-\zeta(x^-)}{y}$$
 are continuous on $[0,b)$ and $(0,b],$ respectively. Clearly, the space $\mathcal{PC}^1_{b}$ is a Banach space with
 \begin{eqnarray*}
 	\| \zeta \|_{\mathcal{PC}^1_b}=\max	\bigg\{\sup\limits_{x \in [-a,b]}\| \zeta \|_{\mathcal{PC}_b}, \sup\limits_{x \in [-a,b]}\| \zeta' \|_{\mathcal{PC}_b}\bigg\}.
 \end{eqnarray*} 
 
 The main contributions of the current work:
 \begin{itemize}
 	\item A class of nonlocal fractional neutral stochastic integrodifferential inclusions of order $1<\alpha<2$ with impulses is introduced.
 	\item We show the existence of mild solution of the system (\ref{1.1}) for both the convex and non-convex values of multivalued maps.
 	\item We establish a set of sufficient conditions that demonstrate the approximate controllability of the system (\ref{1.1}).
 	\item An example is demonstrated to validate the developed theoretical results.
 \end{itemize}
\section{\textbf{Preliminaries and Assumptions}}
\noindent We use the following notations.\\
$P(\mathcal{H})=\{V\in \mathcal{H}: V\neq \emptyset\}$, $P_{cl}(\mathcal{H})=\{V\in P(\mathcal{H}): V ~\mbox{is closed}\}$, $P_{bd}(\mathcal{H})=\{V\in P(\mathcal{H}): V ~\mbox{is bounded}\}$, $P_{cp}(\mathcal{H})=\{V\in P(\mathcal{H}): V ~\mbox{is compact}\}$, $P_{cv}(\mathcal{H})=\{V\in P(\mathcal{H}): V ~\mbox{is convex}\}$, $P_{bd, cl, cv}(\mathcal{H})=\{V\in P(\mathcal{H}): V ~\mbox{is bounded, closed and convex}\}.$\\

\noindent For the sake of convenience in writing, we always set $r=\frac{\alpha}{2},$ for $\alpha\in(1,2).$
\begin{definition} \label{def4} \cite{ZH2019} An $\Upsilon_x$-adapted stochastic process $\zeta: [-a,b] \rightarrow \mathcal{H}$ is called a mild solution of the system (\ref{1.1}) if $\zeta(x)=\phi(x)+h(\zeta)(x),~ x\in[-a,0],~ \zeta'(0)= \xi,$ $u(\cdot) \in L^2([0,b],\mathcal{U}),$ $ \Delta \zeta(x_p)=\hat{I}_p(\zeta(x_p)),~  \Delta \zeta'(x_p)=\hat{J}_p(\zeta(x_p)), ~p=1,2,\ldots, n;$ $\zeta(\cdot)_{[0,b]} \in \mathcal{PC}^1_{b} $ and satisfies
	\begin{eqnarray} \label{2.11}
		\zeta(x)&=&C_{r}(x)[\phi(0)+h(\zeta)]+S_{r}(x)\big[\xi+f_1\big(0,\zeta(0),\zeta(\nu_1(0))\big)\big]\nonumber\\&& \quad -\int_{0}^{x}C_{r}(x-y)f_1\big(y,\zeta(y),\zeta(\nu_1(y))\big)dy +\int_{0}^{x}P_{r}(x-y)\mathcal{B} u(y)dy\nonumber\\&& \quad +\int_{0}^{x}P_{r}(x-y)f_2\big(y,\zeta(y),\zeta(\nu_2(y))\big)dy+\int_{0}^{x}P_{r}(x-y)\int_{0}^{y} \varrho(y-\tau)g(\tau)dw(\tau)dy \nonumber\\&& \quad+\sum \limits_{0 < x_p< x}C_{r}(x-x_p)\hat{I}_p(\zeta(x_p))+\sum \limits_{0 < x_p< x}S_{r}(x-x_p)\hat{J}_p(\zeta(x_p)),
	\end{eqnarray}
	where $g\in S_{\mathcal{G},\zeta}=\big\{g\in L^2(L(\mathcal{W},\mathcal{H})):g(x)\in \mathcal{G}\big(y,\zeta(y),\zeta(\nu_3(y))\big)~ \mbox{a.e.}~ x\in [0,b] \big\}$ is a selection of $\mathcal{G},$ $C_r(x)=\displaystyle\int_0^\infty M_r(\theta)C(x^r\theta)d\theta,$ $S_r(x)=\displaystyle\int_0^xC_r(s)ds$, $P_r(x)=\displaystyle\int_0^\infty r\theta M_r(\theta)S(x^r\theta)d\theta,$ and $M_r(\theta)$ is the Mainardi's Wright-type function.
\end{definition}

From \cite{BM1999}, we introduce the controllability operator as follows:
$$\mu_0^{b}=\displaystyle\int_0^bP_{r}(b-s)\mathcal{B} \mathcal{B}^*P^*_{r}(b-s)ds:\mathcal{H}\to \mathcal{H},$$
where $*$ denotes the adjoint. Clearly, $\mu_0^{b}\geq 0$ and hence, the resolvent
$$\mathcal{R}\big(\delta, -\mu_0^{b}\big)=(\delta I+\mu_0^b)^{-1}:\mathcal{H}\to \mathcal{H},$$ is well defined for all $\delta>0.$\\

Consider $H_d:\mathcal{P}(\mathcal{H})\times\mathcal{P}(\mathcal{H}) \to \mathbb{R}^+\cup\{\infty\}$ given by
\begin{eqnarray*}
	H_d(R,S)=\max\bigg\{\sup\limits_{a\in R}d(a,S),\sup\limits_{b\in S}d(R,b)\bigg\},
\end{eqnarray*}
where $d(R,b)=\inf\limits_{a\in R}d(a,b).$ Then, $(P_{cl}(\mathcal{H}),H_d)$ is a generalized metric space.

\begin{theorem}\label{lm4}\cite{M1975}
	Let $\mathcal{H}$ be a Hilbert space and $\mathcal{K}:\mathcal{H} \to P_{bd, cl, cv}(\mathcal{H})$ be an upper semi continuous (u.s.c.) and condensing map. If 
		$$\mathcal{M}=\{\zeta\in \mathcal{H}:\lambda \zeta \in \mathcal{K}\zeta ~\mbox{for some}~ \lambda>1\},$$ is bounded, then $\mathcal{K}$ has a fixed point.
\end{theorem}
\begin{lemma}\label{lm5}\cite{CN1970} Let $(\mathcal{H},d)$ be a complete metric space. If $\mathcal{K}:\mathcal{H} \to \mathcal{P}_{cl}(\mathcal{H})$ is a contraction, then $\mathcal{K}$ has a fixed point.
	
\end{lemma}

Consider the following assumptions.
\begin{itemize}
	\item [(A1)] The operator $C_r(x),~ x>0$ is compact. Moreover, there exist constants $M_1\geq 1, M_2=M_1b, M_3=\dfrac{M_1b^{r}}{\Gamma(2r)},  M_4>0$ such that
	\begin{eqnarray*}
		\|C_r(x)\|^2\leq M_1, \quad 	\|S_r(x)\|^2\leq M_2, \quad 	\|P_r(x)\|^2\leq M_3,\quad \|\mathcal{B}\|^2\leq M_4.
	\end{eqnarray*}
	\item[(A2)] The function $f_1:[0,b]\times\mathcal{PC}\times \mathcal{PC}_{0} \to \mathcal{H} $ is continuous and there exist positive constants $L_{f_1},k_1, k_2$ such that 
	\begin{itemize}
		\item[(i)] $E\left\|f_1\big(x, \zeta_1(x), \zeta_1(\nu_1(x))\big)-f_1\big(x, \zeta_2(x), \zeta_2(\nu_1(x))\big)\right\|^2\\\leq L_{f_1}E\big(\|\zeta_1(x)-\zeta_2(x)\|^2+\|\zeta_1(\nu_1(x))-\zeta_2(\nu_1(x))\|^2\big),$
		\item[(ii)] $E\left\|f_1\big(x, \zeta(x), \zeta(\nu_1(x))\big)\right\|^2\leq k_1E\big(1+\|\zeta(x)\|^2+\|\zeta(\nu_1(x))\|^2\big),$
		\item[(iii)] $E\|f_1\big(0,\zeta(0),\zeta(\nu_1(0))\big)\|^2=k_2.$
	\end{itemize}
\item[(A3)] The  function $f_2:[0,b]\times\mathcal{PC}\times \mathcal{PC}_{0} \to \mathcal{H} $ satisfies the following
\begin{itemize}
	\item [(i)] $x\mapsto f_2\big(x, \zeta(x), \zeta(\nu_2(x))\big)$ is measurable for $\zeta(x) \in\mathcal{PC} , \zeta(\nu_2(x)) \in \mathcal{PC}_0,$
	\item[(ii)] $\big(\zeta(x), \zeta(\nu_2(x))\big)\mapsto f_2\big(x, \zeta(x), \zeta(\nu_2(x))\big)$ is continuous for almost all $x \in[0,b],$
	\item[(iii)] there exists $L_{f_2}>0$ such that\\
	$E\left\|f_2\big(x, \zeta_1(x), \zeta_1(\nu_2(x))\big)-f_2\big(x, \zeta_2(x), \zeta_2(\nu_2(x))\big)\right\|^2\\\leq L_{f_2}E\big(\|\zeta_1(x)-\zeta_2(x)\|^2+\|\zeta_1(\nu_2(x))-\zeta_2(\nu_2(x))\|^2\big),$
	\item[(iv)] for almost every $x\in[0,b],$ $\zeta(x) \in\mathcal{PC}$ and $ \zeta(\nu_2(x)) \in \mathcal{PC}_0$ such that
	\begin{eqnarray*}
			E\left\|f_2\big(x, \zeta(x), \zeta(\nu_2(x))\big)\right\|^2\leq \Theta(x)\pounds \Big(E\|\zeta(x)\|^2+E\|\zeta(\nu_2(x))\|^2\Big),
	\end{eqnarray*}
	where $\Theta \in L^1([0,b], \mathbb{R}^+)$ and $\pounds:\mathbb{R}^+ \to (0,\infty)$ is continuous and increasing function.
	\end{itemize}
	\item[(A4)] The real valued function $\varrho$ is continuous on $[0,b]$ and there exists $\ell>0$ such that $\|\varrho(x)\|^2\leq \ell$ for $x\in [0,b].$
	\item[(A5)] $\mathcal{G}$ is a multivalued map satisfying $\mathcal{G}:[0,b] \times \mathcal{PC}\times \mathcal{PC}_{0} \to P_{bd, cl, cv}L(\mathcal{W},\mathcal{H})$ is measurable with respect to $x$ for each fixed $\zeta(x) \in\mathcal{PC}$ and $ \zeta(\nu_3(x)) \in \mathcal{PC}_0,$ u.s.c. with respect to $\zeta(x) \in\mathcal{PC}$ and $ \zeta(\nu_3(x)) \in \mathcal{PC}_0$ for each $x \in[0,b]$, and for each fixed $\zeta \in L_2(L(\mathcal{W},\mathcal{H})),$ the selection 
	$$ S_{\mathcal{G},\zeta}=\big\{g\in L^2(L(\mathcal{W},\mathcal{H})):g(x)\in (\hat {\mathcal{G}}\zeta)(x)= \mathcal{G}\big(x,\zeta(x),\zeta(\nu_3(x))\big)~ \mbox{a.e.}~ x\in [0,b] \big\}$$ is nonempty.
	\item[(A6)] There exists $\wp\in L^1([0,b], \mathbb{R}^+)$ such that
	\begin{eqnarray*}
		\int_0^x E\|\mathcal{G}\big(y,\zeta(y),\zeta(\nu_3(y))\big)\|^2_{L_2^0} dy&=&\sup\limits_{x\in[0,b]} \bigg\{\int_0^x E\|g(y)\|^2 dy: g \in(\hat {\mathcal{G}}\zeta)(y) \bigg\}\\&\leq& \wp(x)\beta\Big(E\|\zeta(x)\|^2+E\|\zeta(\nu_3(x))\|^2\Big),
	\end{eqnarray*}
where $\beta:\mathbb{R}^+ \to (0,\infty)$ is continuous and increasing function.
\item[(A7)] The functions $\hat{I}_p, \hat{J}_p:\mathcal{H}\rightarrow \mathcal{H}$ are continuous functions and there exist $L_{\hat I_p}, L_{\hat J_p}>0$ such that
$E\left\|\hat{I}_p(\zeta(x_p))\right\|^2\leq L_{\hat{I}_p}, \quad 	E\left\|\hat{J}_p(\zeta(x_p))\right\|^2\leq L_{\hat{J}_p}$ and
	$$	E\left\|\hat{I}_p(\zeta_1(x_p))-\hat{I}_p(\zeta_2(x_p))\right\|^2\leq L_{\hat{I}_p}\|\zeta_1-\zeta_2\|^2,$$
	$$	E\left\|\hat{J}_p(\zeta_1(x_p))-\hat{J}_p(\zeta_2(x_p))\right\|^2\leq L_{\hat{J}_p}\|\zeta_1-\zeta_2\|^2,$$
for any $\zeta_1, \zeta_2 \in \mathcal{H}$ and $p=1,2,\ldots,n.$
\item[(A8)] $h$ is completely continuous and $L_h>0$ such that
\begin{itemize}
	\item [(i)]
	$E\left\|h(\zeta)-h(\tilde\zeta)\right\|^2\leq L_h \|\zeta-\tilde\zeta\|^2,$
\item[(ii)] $E\left\|h(\zeta)\right\|^2\leq L_h\big(1+ \|\zeta\|^2),$
\end{itemize}
for all $\zeta, \tilde{\zeta} \in \mathcal{H}.$
\item[(A9)] The resolvent operator $(\delta I+\mu_0^b)^{-1}$ satisfies the following conditions:
\begin{itemize} 
	\item[(i)]
	$\big\|(\delta I+\mu_0^b)^{-1}\big\| \leq \frac{1}{\delta},\quad \mbox{for all}~ \delta \geq 0.$
	\item[(ii)] $\delta(\delta I+\mu_0^b)^{-1} \rightarrow 0$ in the strong operator topology as $\delta\rightarrow 0^+.$
\end{itemize}
	\item[(A10)] The multivalued map $\mathcal{G}:[0,b] \times \mathcal{PC}\times \mathcal{PC}_{0} \rightarrow P_{cp}\big(L(\mathcal{W},\mathcal{H})\big)$ has the property that\\ $\mathcal{G}\big(\cdot,\zeta(y),\zeta(\nu_3(y))\big):[0,b] \to P_{cp}\big(L(\mathcal{W},\mathcal{H})\big)$ is measurable for each $\zeta(y)\in \mathcal{PC},~\zeta(\nu_3(y)) \in \mathcal{PC}_{0}.$
	\item[(A11)] There exists $\hat \wp\in L^1([0,b], \mathbb{R}^+)$ such that
	\begin{eqnarray*}
	H_d\big((\hat{\mathcal{G}}\zeta_1)(x),(\hat{\mathcal{G}}\zeta_2)(x)\big)&\leq& \hat \wp(x)\|\zeta_1-\zeta_2\|^2,
	\end{eqnarray*}
and $d\big(0, \mathcal{G}(x,0,0)\big)\leq \hat \wp(x)$ for a.e. $x\in [0,b].$
\end{itemize}

From \cite{BM1999}, it can be observed that (A9)(ii) satisfies if and only if the corresponding linear deterministic system of (\ref{1.1}) is approximately controllable on $[0,b].$

\begin{lemma}\label{lm1} \cite{n11} Let  $\chi: [0,b] \times \Omega \rightarrow L_2^0$ be strongly measurable mapping such that\\ $\displaystyle\int\limits_{0}^{b}{E} \|\chi(s)\|_{L_2^0}^pds < \infty.$ Then
	$${E} \left\| \int\limits_{0}^{t}\chi(s)dw(s) \right \|^p \leq C_p \int\limits_{0}^{t} {E} \|\chi(s)\|_{L_2^0}^pds,$$
	for all $0 \leq t \leq b$ and $p \geq 2,$ where $C_p$ is the constant involving $p$ and $b.$
\end{lemma} 
\begin{lemma}\cite{n11}
	For any $\bar{\zeta}_b \in L^2(\Upsilon_b, \mathcal{H}),$ there exists $\varphi \in L^2_{\Upsilon}\big(\Omega, L^2([0,b],L_2^0)\big)$ such that $\bar{\zeta}_b=E\bar{\zeta}_b+\displaystyle\int_{0}^{b} \varphi(y)dw(y).$
\end{lemma}
For any $\delta>0$ and $\bar{\zeta}_b\in L^2(\Upsilon_b, \mathcal{H})$ the control function is defined as
\begin{eqnarray*}
u_{\delta}(x)&=&\mathcal{B}^*P^*_{r}(b-x)\bigg[\big(\delta I+\mu_0^b\big)^{-1}\bigg(E\bar{\zeta}_b
	-C_{r}(b)[\phi(0)+h(\zeta)]-S_{r}(b)\big[\xi+f_1\big(0,\zeta(0),\zeta(\nu_1(0))\big)\big]\bigg)\\&&\quad+\int_{0}^{b}\big(\delta I+\mu_y^b\big)^{-1} \varphi(y)dw(y)+\int_0^b\big(\delta I+\mu_y^b\big)^{-1}C_{r}(b-y)f_1\big(y,\zeta(y),\zeta(\nu_1(y))\big)dy\\&&\quad-\int_0^b\big(\delta I+\mu_y^b\big)^{-1}P_{r}(b-y)f_2\big(y,\zeta(y),\zeta(\nu_2(y))\big)dy\\&&\quad-\int_{0}^{b}\big(\delta I+\mu_y^b\big)^{-1}P_{r}(b-y)\int_{0}^{y} \varrho(y-\tau)g(\tau)dw(\tau)dy\\&&\quad-\sum \limits_{ p=1}^{n}\big(\delta I+\mu_y^b\big)^{-1}C_{r}(b-x_p)\hat{I}_p(\zeta(x_p)) -\sum \limits_{ p=1}^{n}\big(\delta I+\mu_y^b\big)^{-1}S_{r}(b-x_p)\hat{J}_p(\zeta(x_p))\bigg].
\end{eqnarray*} 
\section{\textbf{Existence of Mild Solutions}}
\begin{theorem}\label{th1} (Convex Case)
	If the hypotheses (A1)-(A8), (A9)(i) are satisfied, then the system (\ref{1.1}) has at least one mild solution.
\end{theorem}
\begin{proof}
	Transform the system (\ref{1.1}) such that it becomes a fixed point problem. Let the multivalued map $\mathcal{K}:\mathcal{PC}^1_b\rightarrow P\big({\mathcal{PC}^1_b}\big)$ and $g\in S_{\mathcal{G},\zeta}$ determined by
	\begin{eqnarray*}
		\big(\mathcal{K}\zeta\big)(x)=	\left\{ \begin{array}{lll}
		\eta \in \mathcal{PC}^1_b:\eta(x)=\left\{ \begin{array}{lll}
		\phi(x)+h(\zeta)(x),\quad x\in[-a,0],
			&\\C_{r}(x)[\phi(0)+h(\zeta)]+S_{r}(x)\big[\xi+f_1\big(0,\zeta(0),\zeta(\nu_1(0))\big)\big] \\-\displaystyle\int_{0}^{x}C_{r}(x-y)f_1\big(y,\zeta(y),\zeta(\nu_1(y))\big)dy+\displaystyle\int_{0}^{x}P_{r}(x-y)\mathcal{B} u(y)dy\\ +\displaystyle\int_{0}^{x}P_{r}(x-y)f_2\big(y,\zeta(y),\zeta(\nu_2(y))\big)dy\\+\displaystyle\int_{0}^{x}P_{r}(x-y)\int_{0}^{y} \varrho(y-\tau)g(\tau)dw(\tau)dy \\+\sum \limits_{0 < x_p< x}C_{r}(x-x_p)\hat{I}_p(\zeta(x_p))\\+\sum \limits_{0 < x_p< x}S_{r}(x-x_p)\hat{J}_p(\zeta(x_p)), \quad x\in[0,b].
		\end{array}	\right.
		\end{array}	\right.
	\end{eqnarray*}
We verify that $\mathcal{K}$ is completely continuous with bounded, closed and convex values, and that it is u.s.c. For convenience, we divide the proof into following steps.\\
\textbf{Step 1.} $\mathcal{K}\zeta$ is convex for every $\zeta \in  \mathcal{PC}^1_b.$ In fact, if $\eta_1, \eta_2 \in \mathcal{K}\zeta,$ then there exist $g_1,g_2 \in S_{\mathcal{G},\zeta} $ such that $x\in [0,b],$ we obtain
\begin{eqnarray*}
	\eta_i(x)&=&C_{r}(x)[\phi(0)+h(\zeta)]+S_{r}(x)\big[\xi+f_1\big(0,\zeta(0),\zeta(\nu_1(0))\big)\big] \\&&-\displaystyle\int_{0}^{x}C_{r}(x-y)f_1\big(y,\zeta(y),\zeta(\nu_1(y))\big)dy+\displaystyle\int_{0}^{x}P_{r}(x-y)\mathcal{B} \mathcal{B}^*P^*_{r}(b-y)\\&&\quad\times\bigg[\big(\delta I+\mu_0^b\big)^{-1}\bigg(E\bar{\zeta}_b
	-C_{r}(b)[\phi(0)+h(\zeta)]-S_{r}(b)\big[\xi+f_1\big(0,\zeta(0),\zeta(\nu_1(0))\big)\big]\bigg)\\&&\quad+\big(\delta I+\mu_y^b\big)^{-1}\bigg\{\int_{0}^{b} \varphi(y)dw(y)+\int_0^bC_{r}(b-y)f_1\big(y,\zeta(y),\zeta(\nu_1(y))\big)dy\\&&\quad-\int_0^bP_{r}(b-y)f_2\big(y,\zeta(y),\zeta(\nu_2(y))\big)dy-\int_{0}^{b}P_{r}(b-y)\int_{0}^{y} \varrho(y-\tau)g_i(\tau)dw(\tau)dy\\&&\quad-\sum \limits_{ p=1}^{n}C_{r}(b-x_p)\hat{I}_p(\zeta(x_p)) -\sum \limits_{ p=1}^{n}S_{r}(b-x_p)\hat{J}_p(\zeta(x_p))\bigg\}\bigg]ds\\&& +\displaystyle\int_{0}^{x}P_{r}(x-y)f_2\big(y,\zeta(y),\zeta(\nu_2(y))\big)dy+\displaystyle\int_{0}^{x}P_{r}(x-y)\int_{0}^{y} \varrho(y-\tau)g_i(\tau)dw(\tau)dy \\&&+\sum \limits_{0 < x_p< x}C_{r}(x-x_p)\hat{I}_p(\zeta(x_p))+\sum \limits_{0 < x_p< x}S_{r}(x-x_p)\hat{J}_p(\zeta(x_p)), \quad i=1,2.
\end{eqnarray*}
Let $0\leq \lambda \leq 1.$ Then for each $x\in [0,b],$ we have
\begin{eqnarray*}
\big(\lambda \eta_1+(1-\lambda)\eta_2\big)(x)&=&C_{r}(x)[\phi(0)+h(\zeta)]+S_{r}(x)\big[\xi+f_1\big(0,\zeta(0),\zeta(\nu_1(0))\big)\big] \\&&-\displaystyle\int_{0}^{x}C_{r}(x-y)f_1\big(y,\zeta(y),\zeta(\nu_1(y))\big)dy\\&&+\displaystyle\int_{0}^{x}P_{r}(x-y)\mathcal{B} \mathcal{B}^*P^*_{r}(b-y)\bigg[\big(\delta I+\mu_0^b\big)^{-1}\bigg(E\bar{\zeta}_b
	-C_{r}(b)[\phi(0)+h(\zeta)]\\&&\quad-S_{r}(b)\big[\xi+f_1\big(0,\zeta(0),\zeta(\nu_1(0))\big)\big]\bigg)+\big(\delta I+\mu_y^b\big)^{-1}\bigg\{\int_{0}^{b} \varphi(y)dw(y)\\&&\quad+\int_0^bC_{r}(b-y)f_1\big(y,\zeta(y),\zeta(\nu_1(y))\big)dy\\&&\quad-\int_0^bP_{r}(b-y)f_2\big(y,\zeta(y),\zeta(\nu_2(y))\big)dy\\&&\quad-\int_{0}^{b}P_{r}(b-y)\int_{0}^{y} \varrho(y-\tau)[\lambda g_1(\tau)+(1-\lambda)g_2(\tau)]dw(\tau)dy\\&&\quad-\sum \limits_{ p=1}^{n}C_{r}(b-x_p)\hat{I}_p(\zeta(x_p)) -\sum \limits_{ p=1}^{n}S_{r}(b-x_p)\hat{J}_p(\zeta(x_p))\bigg\}\bigg]ds\\&& +\displaystyle\int_{0}^{x}P_{r}(x-y)f_2\big(y,\zeta(y),\zeta(\nu_2(y))\big)dy\\&&+\displaystyle\int_{0}^{x}P_{r}(x-y)\int_{0}^{y} \varrho(y-\tau)[\lambda g_1(\tau)+(1-\lambda)g_2(\tau)]dw(\tau)dy \\&&+\sum \limits_{0<x_p<x}C_{r}(x-x_p)\hat{I}_p(\zeta(x_p))+\sum \limits_{0<x_p<x}S_{r}(x-x_p)\hat{J}_p(\zeta(x_p)).
\end{eqnarray*}
Clearly, $S_{\mathcal{G},\zeta}$ is convex, since $\mathcal{G}$ is convex values. Therefore, $\lambda g_1+(1-\lambda)g_2\in S_{\mathcal{G},\zeta}.$ Consequently, 
$\lambda \eta_1+(1-\lambda)\eta_2\in \mathcal{K}\zeta$.\\
\noindent \textbf{Step 2.} $\mathcal{K}$ maps bounded sets into bounded sets in $\mathcal{PC}^1_b.$

\noindent In fact, it is enough to prove that there exists a constant $\epsilon>0$ such that for any $\eta \in \mathcal{K}\zeta,~ \zeta \in \mathcal{Z}_q,$ where
\begin{eqnarray*}
	\mathcal{Z}_q=\big\{\zeta \in\mathcal{PC}^1_b: \|\zeta\|^2_{\mathcal{PC}^1_b}\leq q, ~\mbox{where}~ q=\max\{q_1, q_2 \}~ \mbox {such that}~\|\zeta\|^2_{\mathcal{PC}_b}\leq q_1, \|\zeta'\|^2_{\mathcal{PC}^1_b}\leq q_2\big\}, 
\end{eqnarray*}
one possesses $\|\eta\|^2_{\mathcal{PC}^1_b} \leq \epsilon.$
If $\eta \in \mathcal{K}\zeta,$ then there exists $g\in  S_{\mathcal{G},\zeta}$ such that for any $x\in [0,b],$ we have
\begin{eqnarray*}
	\eta(x)&=&C_{r}(x)[\phi(0)+h(\zeta)]+S_{r}(x)\big[\xi+f_1\big(0,\zeta(0),\zeta(\nu_1(0))\big)\big] \\&&-\displaystyle\int_{0}^{x}C_{r}(x-y)f_1\big(y,\zeta(y),\zeta(\nu_1(y))\big)dy+\displaystyle\int_{0}^{x}P_{r}(x-y)\mathcal{B} u_\delta (y)dy\\&& +\displaystyle\int_{0}^{x}P_{r}(x-y)f_2\big(y,\zeta(y),\zeta(\nu_2(y))\big)dy+\displaystyle\int_{0}^{x}P_{r}(x-y)\int_{0}^{y} \varrho(y-\tau)g(\tau)dw(\tau)dy \\&&+\sum \limits_{0<x_p<x}C_{r}(x-x_p)\hat{I}_p(\zeta(x_p))+\sum \limits_{0<x_p<x}S_{r}(x-x_p)\hat{J}_p(\zeta(x_p)).
\end{eqnarray*}
Hence, from the hypotheses and Lemma~\ref{lm1}, it follows that
\begin{eqnarray*}
	E\left\|\eta(x)\right\|^2&\leq&8E\left\|C_{r}(x)[\phi(0)+h(\zeta)]\right\|^2+8E\left\|S_{r}(x)\big[\xi+f_1\big(0,\zeta(0),\zeta(\nu_1(0))\big)\big]\right\|^2 \\&&+8E\left\|\displaystyle\int_{0}^{x}C_{r}(x-y)f_1\big(y,\zeta(y),\zeta(\nu_1(y))\big)dy\right\|^2+8E\left\|\displaystyle\int_{0}^{x}P_{r}(x-y)\mathcal{B} u_\delta (y)dy\right\|^2\\&& +8E\left\|\displaystyle\int_{0}^{x}P_{r}(x-y)f_2\big(y,\zeta(y),\zeta(\nu_2(y))\big)dy\right\|^2\\&&+8E\left\|\displaystyle\int_{0}^{x}P_{r}(x-y)\int_{0}^{y} \varrho(y-\tau)g(\tau)dw(\tau)dy\right\|^2 \\&&+8E\Bigg\|\sum \limits_{0<x_p<x}C_{r}(x-x_p)\hat{I}_p(\zeta(x_p))\Bigg\|^2+8E\Bigg\|\sum \limits_{0<x_p<x}S_{r}(x-x_p)\hat{J}_p(\zeta(x_p))\Bigg\|^2\\
	&\leq&16\left\|C_{r}(x)\right\|^2E\big[\|\phi(0)\|^2+\|h(\zeta)\|^2\big] \\&&+16\left\|S_{r}(x)\right\|^2E\big[\|\xi\|^2+\big\|f_1\big(0,\zeta(0),\zeta(\nu_1(0))\big)\big\|^2\big]\\&& 
	+8\displaystyle\int_{0}^{x}\|C_{r}(x-y)\|^2  E\big\|f_1\big(y,\zeta(y),\zeta(\nu_1(y))\big)\big\|^2dy\\&&+64\displaystyle\int_{0}^{x}\|P_{r}(x-y)\|^2 \|\mathcal{B}\|^2\|\mathcal{B}^*\|^2\|P^*_{r}(b-y)\|^2\bigg[\left\|\big(\delta I+\mu_0^b\big)^{-1}\right\|^2\bigg(2E\|\bar{\zeta}_b\|^2\\&&\quad+2\|C_{r}(b)\|^2E\big[\|\phi(0)\|^2+\|h(\zeta)\|^2\big]\\&&\quad+2\|S_{r}(b)\|^2E\big[\|\xi\|^2+\big\|f_1\big(0,\zeta(0),\zeta(\nu_1(0))\big)\big\|^2\big]\bigg)\\&&\quad+\left\|\big(\delta I+\mu_y^b\big)^{-1}\right\|^2\bigg\{2C_2\int_{0}^{b} E\left\|\varphi(y)\right\|^2_{L_2^0}dy\\&&\quad+\int_0^b\|C_{r}(b-y)\|^2E\left\|f_1\big(y,\zeta(y),\zeta(\nu_1(y))\big)\right\|^2dy\\&&\quad+\int_0^b\|P_{r}(b-y)\|^2E\left\|f_2\big(y,\zeta(y),\zeta(\nu_2(y))\big)\right\|^2dy\\&&\quad+C_2\int_{0}^{b}\|P_{r}(b-y)\|^2\int_{0}^{y} \|\varrho(y-\tau)\|^2E\left\|g(\tau)\right\|^2d\tau dy\\&&\quad+\sum \limits_{p=1}^n\|C_{r}(b-x_p)\|^2E\|\hat{I}_p(\zeta(x_p))\|^2+\sum \limits_{p=1}^n\|S_{r}(b-x_p)\|^2E\|\hat{J}_p(\zeta(x_p))\|^2\bigg\}\bigg]\\&& +8\displaystyle\int_{0}^{x}\|P_{r}(x-y)\|^2E\left\|f_2\big(y,\zeta(y),\zeta(\nu_2(y))\big)\right\|^2dy\\&&+8C_2\displaystyle\int_{0}^{x}\|P_{r}(x-y)\|^2\int_{0}^{y} \|\varrho(y-\tau)\|^2E\left\|g(\tau)\right\|^2d\tau dy \\&&+8\sum \limits_{0<x_p<x}\left\|C_{r}(x-x_p)\right\|^2E\left\|\hat{I}_p(\zeta(x_p))\right\|^2+8\sum \limits_{0<x_p<x}\left\|S_{r}(x-x_p)\right\|^2E\left\|\hat{J}_p(\zeta(x_p))\right\|^2\\
	&\leq& 16 M_1\big[E\|\phi(0)\|^2+L_h(1+q)\big]+16M_2\big[E\|\xi\|^2+k_2\big]+8M_1bk_1\big(1+2q\big)\\&&+\frac{64}{\delta}M_3^2M_4^2b\bigg[2E\|\bar{\zeta}_b\|^2+2M_1\big[E\|\phi(0)\|^2+L_h(1+q)\big]+2M_2\big[E\|\xi\|^2+k_2\big]\\&&\quad+2C_2\int_{0}^{b} E\left\|\varphi(y)\right\|^2_{L_2^0}dy+M_1bk_1\big(1+2q\big)+M_3\int_{0}^{b}\Theta(y)\pounds(2q)dy\\&&\quad+C_2M_3\ell\beta(2q)\int_0^b\wp(y)dy+\sum \limits_{p=1}^nM_1L_{\hat{I}_p}+\sum \limits_{p=1}^nM_2L_{\hat{J}_p}\bigg]\\&&+8M_3\pounds(2q)\int_0^x\Theta(y)dy+8C_2M_3\ell\beta(2q)\displaystyle\int_{0}^{x}\wp(y) dy+8\sum \limits_{0<x_p<x}M_1L_{\hat{I}_p}\\&&+8\sum \limits_{0<x_p<x}M_2L_{\hat{J}_p}\\
	&\leq& \epsilon.
\end{eqnarray*}
Then, for each $\eta \in \mathcal{K} \zeta,$ we get $\|\eta\|^2_{\mathcal{PC}^1_b} \leq \epsilon.$\\
\textbf{Step 3.} $\mathcal{K}$ maps bounded sets into equicontinuous sets of $\mathcal{PC}^1_b.$\\
\noindent For any $\eta \in \mathcal{K}\zeta$ and $\zeta \in \mathcal{Z}_q,$ there exists $g\in S_{\mathcal{G},\zeta} $ such that
\begin{eqnarray*}
	\eta(x)&=&C_{r}(x)[\phi(0)+h(\zeta)]+S_{r}(x)\big[\xi+f_1\big(0,\zeta(0),\zeta(\nu_1(0))\big)\big] \\&&-\displaystyle\int_{0}^{x}C_{r}(x-y)f_1\big(y,\zeta(y),\zeta(\nu_1(y))\big)dy+\displaystyle\int_{0}^{x}P_{r}(x-y)\mathcal{B} u_\delta (y)dy\\&& +\displaystyle\int_{0}^{x}P_{r}(x-y)f_2\big(y,\zeta(y),\zeta(\nu_2(y))\big)dy\\&&+\displaystyle\int_{0}^{x}P_{r}(x-y)\int_{0}^{y} \varrho(y-\tau)g(\tau)dw(\tau)dy \\&&+\sum \limits_{0<x_p<x}C_{r}(x-x_p)\hat{I}_p(\zeta(x_p))+\sum \limits_{0<x_p<x}S_{r}(x-x_p)\hat{J}_p(\zeta(x_p)).
\end{eqnarray*}
Let $0<x_1<x_2\leq b,$ then we get\\
$E\|\eta(x_2)-\eta(x_1)\|^2$
\begin{eqnarray*}
	\hspace{-2cm}&\leq&14E\left\|[C_{r}(x_2)-C_{r}(x_1)][\phi(0)+h(\zeta)]\right\|^2\\&&+14E\big\|\big[S_{r}(x_2)-S_{r}(x_1)\big]\big[\xi+f_1\big(0,\zeta(0),\zeta(\nu_1(0))\big)\big]\big\|^2 \\&&+14E\left\|\displaystyle\int_{0}^{x_1}\big[C_{r}(x_2-y)-C_{r}(x_1-y)\big]f_1\big(y,\zeta(y),\zeta(\nu_1(y))\big)dy\right\|^2\\&&+14E\left\|\displaystyle\int_{x_1}^{x_2}C_{r}(x_2-y)f_1\big(y,\zeta(y),\zeta(\nu_1(y))\big)dy\right\|^2\\&&
	+14E\left\|\displaystyle\int_{0}^{x_1}\big[P_{r}(x_2-y)-P_{r}(x_1-y)\big]\mathcal{B} u_\delta (y)dy\right\|^2\\&&
	+14E\left\|\displaystyle\int_{x_1}^{x_2}P_{r}(x_2-y)\mathcal{B} u_\delta (y)dy\right\|^2\\&& +14E\left\|\displaystyle\int_{0}^{x_1}\big[P_{r}(x_2-y)-P_{r}(x_1-y)\big]f_2\big(y,\zeta(y),\zeta(\nu_2(y))\big)dy\right\|^2\\&&+14E\left\|\displaystyle\int_{x_1}^{x_2}P_{r}(x_2-y)f_2\big(y,\zeta(y),\zeta(\nu_2(y))\big)dy\right\|^2\\&&+14E\left\|\displaystyle\int_{0}^{x_1}\big[P_{r}(x_2-y)-P_{r}(x_1-y)\big]\int_{0}^{y} \varrho(y-\tau)g(\tau)dw(\tau)dy\right\|^2 \\&&+14E\left\|\displaystyle\int_{x_1}^{x_2}P_{r}(x_2-y)\int_{0}^{y} \varrho(y-\tau)g(\tau)dw(\tau)dy\right\|^2 \\&&+14E\left\|\sum \limits_{0 < x_p < x_1}\big[C_{r}(x_2-x_p)-C_{r}(x_1-x_p)\big]\hat{I}_p(\zeta(x_p))\right\|^2\\&&+14E\left\|\sum \limits_{x_1 \leq x_p < x_2}C_{r}(x_2-x_p)\hat{I}_p(\zeta(x_p))\right\|^2\\&&+14E\left\|\sum \limits_{0 < x_p < x_1}\big[S_{r}(x_2-x_p)-S_{r}(x_1-x_p)\big]\hat{J}_p(\zeta(x_p))\right\|^2\\&&+14E\left\|\sum \limits_{x_1 \leq x_p < x_2}S_{r}(x_2-x_p)\hat{J}_p(\zeta(x_p))\right\|^2\\
	&\leq &28\left\|C_{r}(x_2)-C_{r}(x_1)\right\|^2\big[E\|\phi(0)\|^2+L_h(1+q)\big]\\&&+28\big\|S_{r}(x_2)-S_{r}(x_1)\big\|^2\big[E\|\xi\|^2+k_2\big]\\&&+14\displaystyle\int_{0}^{x_1}\left\|C_{r}(x_2-y)-C_{r}(x_1-y)\right\|^2k_1\big(1+2q\big)dy\\&&+14\displaystyle\int_{x_1}^{x_2}M_1k_1\big(1+2q\big)dy+\frac{112}{\delta}M_3M_4^2\displaystyle\int_{0}^{x_1}\|P_{r}(x_2-y)-P_{r}(x_1-y)\|^2 \\&&\quad\times\bigg[2E\|\bar{\zeta}_b\|^2+2C_2\int_{0}^{b} E\left\|\varphi(y)\right\|^2_{L_2^0}dy+2M_1\big[E\|\phi(0)\|^2+L_h(1+q)\big]\\&&\quad\quad+2M_2\big[E\|\xi\|^2+k_2\big]+M_1bk_1\big(1+2q\big)+M_3\int_{0}^b\Theta(y)\pounds(2q)dy\\&&\quad\quad+C_2M_3\ell\beta(2q)\int_0^b\wp(y)dy+\sum \limits_{p=1}^nM_1L_{\hat{I}_p}+\sum \limits_{p=1}^nM_2L_{\hat{J}_p}\bigg]dy
	\\&&+\frac{112}{\delta}M_3^2M_4^2\displaystyle\int_{x_1}^{x_2}\bigg[2E\|\bar{\zeta}_b\|^2+2C_2\int_{0}^{b} E\left\|\varphi(y)\right\|^2_{L_2^0}dy\\&&\quad+2M_1\big[E\|\phi(0)\|^2+L_h(1+q)\big]+2M_2\big[E\|\xi\|^2+k_2\big]+M_1bk_1\big(1+2q\big)\\&&\quad+M_3\int_0^b\Theta(y)\pounds(2q)dy+C_2M_3\ell\beta(2q)\int_0^b\wp(y)dy+\sum \limits_{p=1}^nM_1L_{\hat{I}_p}+\sum \limits_{p=1}^nM_2L_{\hat{J}_p}\bigg]dy\\&&
 +14\displaystyle\int_{0}^{x_1}\left\|P_{r}(x_2-y)-P_{r}(x_1-y)\right\|^2\pounds(2q)\Theta(y)dy\\&&+14\displaystyle\int_{x_1}^{x_2}M_3\pounds(2q)\Theta(y)dy\\&&+14\displaystyle\int_{0}^{x_1}\left\|P_{r}(x_2-y)-P_{r}(x_1-y)\right\|^2C_2\ell\beta(2q)\wp(y)dy \\&&+14\displaystyle\int_{x_1}^{x_2}C_2M_3\ell\beta(2q)\wp(y)dy \\&&+14\sum \limits_{0< x_p< x_1}\left\|C_{r}(x_2-x_p)-C_{r}(x_1-x_p)\right\|^2L_{\hat{I}_p}\\&&+14\sum \limits_{x_1\leq x_p <x_2}\left\|C_{r}(x_2-x_p)\right\|^2L_{\hat{I}_p}\\&&+14\sum \limits_{0<x_p <x_1}\left\|S_{r}(x_2-x_p)-S_{r}(x_1-x_p)\right\|^2L_{\hat{J}_p}\\&&+14\sum \limits_{x_1\leq x_p <x_2}\left\|S_{r}(x_2-x_p)\right\|^2L_{\hat{J}_p}.
\end{eqnarray*} 
The right-hand side of the above inequality is independent of $\zeta \in \mathcal{Z}_q$ and $	E\|\eta(x_2)-\eta(x_1)\|^2 \to 0$ as $x_2-x_1 \to 0,$ for all $\zeta \in \mathcal{Z}_q.$ Thus, the compactness of $C_{r}(x)$ and $S_{r}(x)$ for $x>0$ gives the continuity in the uniform operator topology. Hence, the set $\{\mathcal{K}( \mathcal{Z}_q)\}$ is equicontinuous on $[0,b].$\\

\noindent
\textbf{Step 4.} $\mathcal{K}$ has a closed graph.\\
\noindent Consider $\zeta_n \to \zeta_*, ~\eta_n \in \mathcal{K}\zeta_n$ and $\eta_n \to \eta_*.$ We prove that $\eta_* \in \mathcal{K}\zeta_*, ~\eta_n \in \mathcal{K}\zeta_n$ means that there exists $g_n \in S_{\mathcal{G},\zeta_n} $ such that
\begin{eqnarray}\label{eq4}
	\eta_n(x)&=&C_{r}(x)[\phi(0)+h(\zeta_n)]+S_{r}(x)\big[\xi+f_1\big(0,\zeta_n(0),\zeta_n(\nu_1(0))\big)\big]\nonumber \\&&-\displaystyle\int_{0}^{x}C_{r}(x-y)f_1\big(y,\zeta_n(y),\zeta_n(\nu_1(y))\big)dy\nonumber\\&&+\displaystyle\int_{0}^{x}P_{r}(x-y)\mathcal{B} \mathcal{B}^*P^*_{r}(b-y)\bigg[\big(\delta I+\mu_0^b\big)^{-1}\bigg(E\bar{\zeta}_b
	-C_{r}(b)[\phi(0)+h(\zeta_n)]\nonumber\\&&\quad-S_{r}(b)\big[\xi+f_1\big(0,\zeta_n(0),\zeta_n(\nu_1(0))\big)\big]\bigg)+\big(\delta I+\mu_y^b\big)^{-1}\bigg\{\int_{0}^{b} \varphi(y)dw(y)\nonumber\\&&\quad+\int_0^bC_{r}(b-y)f_1\big(y,\zeta_n(y),\zeta_n(\nu_1(y))\big)dy\notag\\&&\quad-\int_0^bP_{r}(b-y)f_2\big(y,\zeta_n(y),\zeta_n(\nu_2(y))\big)dy\nonumber\\&&\quad-\int_{0}^{b}P_{r}(b-y)\int_{0}^{y} \varrho(y-\tau)g_n(\tau)dw(\tau)dy\nonumber\\&&\quad-\sum \limits_{p=1}^nC_{r}(b-x_p)\hat{I}_p(\zeta_n(x_p)) -\sum \limits_{p=1}^nS_{r}(b-x_p)\hat{J}_p(\zeta_n(x_p))\bigg\}\bigg]dy \nonumber\\&&+\displaystyle\int_{0}^{x}P_{r}(x-y)f_2\big(y,\zeta_n(y),\zeta_n(\nu_2(y))\big)dy\nonumber\\&&+\displaystyle\int_{0}^{x}P_{r}(x-y)\int_{0}^{y} \varrho(y-\tau)g_n(\tau)dw(\tau)dy+\sum \limits_{0<x_p<x}C_{r}(x-x_p)\hat{I}_p(\zeta_n(x_p))\nonumber \\&&+\sum \limits_{0<x_p<x}S_{r}(x-x_p)\hat{J}_p(\zeta_n(x_p)),\quad x\in[0,b].
\end{eqnarray}
We need to show that there exists $g_*\in S_{\mathcal{G},\zeta_*},$ for every $x\in [0,b]$ such that
\begin{eqnarray*}
	\eta_*(x)&=&C_{r}(x)[\phi(0)+h(\zeta_*)]+S_{r}(x)\big[\xi+f_1\big(0,\zeta_*(0),\zeta_*(\nu_1(0))\big)\big] \\&&-\displaystyle\int_{0}^{x}C_{r}(x-y)f_1\big(y,\zeta_*(y),\zeta_*(\nu_1(y))\big)dy\\&&+\displaystyle\int_{0}^{x}P_{r}(x-y)\mathcal{B} \mathcal{B}^*P^*_{r}(b-y)\bigg[\big(\delta I+\mu_0^b\big)^{-1}\bigg(E\bar{\zeta}_b
	-C_{r}(b)[\phi(0)+h(\zeta_*)]\\&&\quad-S_{r}(b)\big[\xi+f_1\big(0,\zeta_*(0),\zeta_*(\nu_1(0))\big)\big]\bigg)+\big(\delta I+\mu_y^b\big)^{-1}\bigg\{\int_{0}^{b} \varphi(y)dw(y)\\&&\quad+\int_0^bC_{r}(b-y)f_1\big(y,\zeta_*(y),\zeta_*(\nu_1(y))\big)dy\\&&\quad-\int_0^bP_{r}(b-y)f_2\big(y,\zeta_*(y),\zeta_*(\nu_2(y))\big)dy\\&&\quad-\int_{0}^{b}P_{r}(b-y)\int_{0}^{y} \varrho(y-\tau)g_*(\tau)dw(\tau)dy\\&&\quad -\sum \limits_{p=1}^nC_{r}(b-x_p)\hat{I}_p(\zeta_*(x_p))-\sum \limits_{p=1}^nS_{r}(b-x_p)\hat{J}_p(\zeta_*(x_p))\bigg\}\bigg]dy \\&&+\displaystyle\int_{0}^{x}P_{r}(x-y)f_2\big(y,\zeta_*(y),\zeta_*(\nu_2(y))\big)dy\\&&+\displaystyle\int_{0}^{x}P_{r}(x-y)\int_{0}^{y} \varrho(y-\tau)g_*(\tau)dw(\tau)dy +\sum \limits_{0<x_p<x}C_{r}(x-x_p)\hat{I}_p(\zeta_*(x_p))\\&&+\sum \limits_{0<x_p<x}S_{r}(x-x_p)\hat{J}_p(\zeta_*(x_p)),\quad x\in[0,b].
\end{eqnarray*}
Clearly, since $\hat{I}_p, \hat{J}_p,~p=1,2,\ldots,n$ are continuous, we have
\begin{eqnarray*}
&&	E\bigg\|\bigg[\eta_n(x)-C_{r}(x)[\phi(0)+h(\zeta_n)]-S_{r}(x)\big[\xi+f_1\big(0,\zeta_n(0),\zeta_n(\nu_1(0))\big)\big] \\&&+\displaystyle\int_{0}^{x}C_{r}(x-y)f_1\big(y,\zeta_n(y),\zeta_n(\nu_1(y))\big)dy-\displaystyle\int_{0}^{x}P_{r}(x-y)f_2\big(y,\zeta_n(y),\zeta_n(\nu_2(y))\big)dy\\&&-\sum \limits_{0<x_p<x}C_{r}(x-x_p)\hat{I}_p(\zeta_n(x_p))-\sum \limits_{0<x_p<x}S_{r}(x-x_p)\hat{J}_p(\zeta_n(x_p))\\&&-\displaystyle\int_{0}^{x}P_{r}(x-y)\mathcal{B} \mathcal{B}^*P^*_{r}(b-y)\bigg\{\big(\delta I+\mu_0^b\big)^{-1}\bigg(E\bar{\zeta}_b
	-C_{r}(b)[\phi(0)+h(\zeta_n)]\\&&-S_{r}(b)\big[\xi+f_1\big(0,\zeta_n(0),\zeta_n(\nu_1(0))\big)\big]\bigg)+\big(\delta I+\mu_y^b\big)^{-1}\bigg(\int_{0}^{b} \varphi(y)dw(y)\\&&+\int_0^bC_{r}(b-y)f_1\big(y,\zeta_n(y),\zeta_n(\nu_1(y))\big)dy-\int_0^bP_{r}(b-y)f_2\big(y,\zeta_n(y),\zeta_n(\nu_2(y))\big)dy\\&&-\sum \limits_{p=1}^nC_{r}(b-x_p)\hat{I}_p(\zeta_n(x_p)) -\sum \limits_{p=1}^nS_{r}(b-x_p)\hat{J}_p(\zeta_n(x_p))\bigg)\bigg\}dy\bigg]\\&&-\bigg[\eta_*(x)-C_{r}(x)[\phi(0)+h(\zeta_*)]-S_{r}(x)\big[\xi+f_1\big(0,\zeta_*(0),\zeta_*(\nu_1(0))\big)\big] \\&&+\displaystyle\int_{0}^{x}C_{r}(x-y)f_1\big(y,\zeta_*(y),\zeta_*(\nu_1(y))\big)dy\\&&-\displaystyle\int_{0}^{x}P_{r}(x-y)f_2\big(y,\zeta_*(y),\zeta_*(\nu_2(y))\big)dy-\sum \limits_{0<x_p<x}C_{r}(x-x_p)\hat{I}_p(\zeta_*(t_p))\\&&-\sum \limits_{0<x_p<x}S_{r}(x-x_p)\hat{J}_p(\zeta_*(x_p))-\displaystyle\int_{0}^{x}P_{r}(x-y)\mathcal{B} \mathcal{B}^*P^*_{r}(b-y)\bigg\{\big(\delta I+\mu_0^b\big)^{-1}\bigg(E\bar{\zeta}_b\\&&
	-C_{r}(b)[\phi(0)+h(\zeta_*)]-S_{r}(b)\big[\xi+f_1\big(0,\zeta_*(0),\zeta_*(\nu_1(0))\big)\big]\bigg)\\&&+\big(\delta I+\mu_y^b\big)^{-1}\bigg(\int_{0}^{b} \varphi(y)dw(y)+\int_0^bC_{r}(b-y)f_1\big(y,\zeta_*(y),\zeta_*(\nu_1(y))\big)dy\\&&-\int_0^bP_{r}(b-y)f_2\big(y,\zeta_*(y),\zeta_*(\nu_2(y))\big)dy-\sum \limits_{p=1}^nC_{r}(b-x_p)\hat{I}_p(\zeta_*(x_p))\\&& -\sum \limits_{p=1}^nS_{r}(b-x_p)\hat{J}_p(\zeta_*(x_p))\bigg)\bigg\}dy\bigg]\bigg\|^2 \to 0 ~\mbox{as}~n \to \infty.
\end{eqnarray*}
Let the linear operator $\Sigma:L^2([0,b],\mathcal{H}) \to C([0,b], \mathcal{H})$
\begin{eqnarray*}
g \to \big(\Sigma g\big)(x)&=&	\displaystyle\int_{0}^{x}P_{r}(x-y)\int_{0}^{y} \varrho(y-\tau)g(\tau)dw(\tau)dy\\&&-\displaystyle\int_{0}^{x}P_{r}(x-y)\mathcal{B} \mathcal{B}^*P^*_{r}(b-y)\big(\delta I+\mu_y^b\big)^{-1}\bigg(\int_{0}^{b}P_{r}(b-y)\\&&\quad\times\int_{0}^{y} \varrho(y-\tau)g(\tau)dw(\tau)dy\bigg)dy.
\end{eqnarray*}
From \cite[Lemma~2.7]{SU2022}, $\Sigma \circ S_{\mathcal{G}}$ is a closed graph operator, we have
\begin{eqnarray*}
	&&	\eta_n(x)-C_{r}(x)[\phi(0)+h(\zeta_n)]-S_{r}(x)\big[\xi+f_1\big(0,\zeta_n(0),\zeta_n(\nu_1(0))\big)\big] \\&&+\displaystyle\int_{0}^{x}C_{r}(x-y)f_1\big(y,\zeta_n(y),\zeta_n(\nu_1(y))\big)dy-\displaystyle\int_{0}^{x}P_{r}(x-y)f_2\big(y,\zeta_n(y),\zeta_n(\nu_2(y))\big)dy\\&&-\sum \limits_{0<x_p<x}C_{r}(x-x_p)\hat{I}_p(\zeta_n(x_p))-\sum \limits_{0<x_p<x}S_{r}(x-x_p)\hat{J}_p(\zeta_n(x_p))\\&&-\displaystyle\int_{0}^{x}P_{r}(x-y)\mathcal{B} \mathcal{B}^*P^*_{r}(b-y)\bigg\{\big(\delta I+\mu_0^b\big)^{-1}\bigg(E\bar{\zeta}_b
	-C_{r}(b)[\phi(0)+h(\zeta_n)]\\&&-S_{r}(b)\big[\xi+f_1\big(0,\zeta_n(0),\zeta_n(\nu_1(0))\big)\big]\bigg)+\big(\delta I+\mu_y^b\big)^{-1}\bigg(\int_{0}^{b} \varphi(y)dw(y)\\&&+\int_0^bC_{r}(b-y)f_1\big(y,\zeta_n(y),\zeta_n(\nu_1(y))\big)dy-\int_0^bP_{r}(b-y)f_2\big(y,\zeta_n(y),\zeta_n(\nu_2(y))\big)dy\\&&-\sum \limits_{p=1}^nC_{r}(b-x_p)\hat{I}_p(\zeta_n(x_p)) -\sum \limits_{p=1}^nS_{r}(b-x_p)\hat{J}_p(\zeta_n(x_p))\bigg)\bigg\}dy \in \Sigma \big( S_{\mathcal{G}, \zeta_n}\big).
\end{eqnarray*}
Since $\zeta_n \to \zeta_*,$ it follows from \cite[Lemma~2.7]{SU2022}, 
\begin{eqnarray*}
&&	\eta_*(x)-C_{r}(x)[\phi(0)+h(\zeta_*)]-S_{r}(x)\big[\xi+f_1\big(0,\zeta_*(0),\zeta_*(\nu_1(0))\big)\big] \\&&+\displaystyle\int_{0}^{x}C_{r}(x-y)f_1\big(y,\zeta_*(y),\zeta_*(\nu_1(y))\big)dy-\displaystyle\int_{0}^{x}P_{r}(x-y)f_2\big(y,\zeta_*(y),\zeta_*(\nu_2(y))\big)dy\\&&-\sum \limits_{0<x_p<x}C_{r}(x-x_p)\hat{I}_p(\zeta_*(x_p))-\sum \limits_{0<x_p<x}S_{r}(x-x_p)\hat{J}_p(\zeta_*(x_p))\\&&-\displaystyle\int_{0}^{x}P_{r}(x-y)\mathcal{B} \mathcal{B}^*P^*_{r}(b-y)\bigg\{\big(\delta I+\mu_0^b\big)^{-1}\bigg(E\bar{\zeta}_b
	-C_{r}(b)[\phi(0)+h(\zeta_*)]\\&&-S_{r}(b)\big[\xi+f_1\big(0,\zeta_*(0),\zeta_*(\nu_1(0))\big)\big]\bigg)+\big(\delta I+\mu_y^b\big)^{-1}\bigg(\int_{0}^{b} \varphi(y)dw(y)\\&&+\int_0^bC_{r}(b-y)f_1\big(y,\zeta_*(y),\zeta_*(\nu_1(y))\big)dy-\int_0^bP_{r}(b-y)f_2\big(y,\zeta_*(y),\zeta_*(\nu_2(y))\big)dy\\&&-\sum \limits_{p=1}^nC_{r}(b-x_p)\hat{I}_p(\zeta_*(x_p)) -\sum \limits_{p=1}^nS_{r}(b-x_p)\hat{J}_p(\zeta_*(x_p))\bigg)\bigg\}dy\\
	&=&\displaystyle\int_{0}^{x}P_{r}(x-y)\int_{0}^{y} \varrho(y-\tau)g_*(\tau)dw(\tau)dy\\&&-\displaystyle\int_{0}^{x}P_{r}(x-y)\mathcal{B} \mathcal{B}^*P^*_{r}(b-y)\big(\delta I+\mu_y^b\big)^{-1}\bigg(\int_{0}^{b}P_{r}(b-y)\int_{0}^{y} \varrho(y-\tau)g_*(\tau)dw(\tau)dy\bigg)dy,
\end{eqnarray*}
for some $g_* \in S_{\mathcal{G},\zeta_*}.$ Therefore, $\mathcal{K}$ has a closed graph.\\
\noindent Hence, $\mathcal{K}$ is completely continuous multivalued map, u.s.c. with convex closed values. To apply Theorem~\ref{lm4}, we need the following step.\\
\textbf{Step 5.} We show that the set $\mathcal{Y}$ is bounded, where
\begin{eqnarray*}
	\mathcal{Y}=\big\{\zeta \in \mathcal{PC}^1_b: \lambda \zeta \in \mathcal{K}\zeta~ \mbox{for some} ~\lambda>1\big\}.
\end{eqnarray*}
Let $\zeta \in \mathcal{Y}.$ In addition, $\lambda \zeta \in \mathcal{K}\zeta$ for some $\lambda>1.$ So, there exists $g\in S_{\mathcal{G},\zeta}$ such that
	\begin{eqnarray*}
	\zeta(x)&=&\lambda^{-1}C_{r}(x)[\phi(0)+h(\zeta)]+\lambda^{-1}S_{r}(x)\big[\xi+f_1\big(0,\zeta(0),\zeta(\nu_1(0))\big)\big]\\&& -\lambda^{-1}\int_{0}^{x}C_{r}(x-y)f_1\big(y,\zeta(y),\zeta(\nu_1(y))\big)dy\nonumber\\&&  +\lambda^{-1}\int_{0}^{x}P_{r}(x-y)\mathcal{B} u(y)dy +\lambda^{-1}\int_{0}^{x}P_{r}(x-y)f_2\big(y,\zeta(y),\zeta(\nu_2(y))\big)dy\nonumber\\&& +\lambda^{-1}\int_{0}^{x}P_{r}(x-y)\int_{0}^{y} \varrho(y-\tau)g(\tau)dw(\tau)dy +\lambda^{-1}\sum \limits_{0 < x_p< x}C_{r}(x-x_p)\hat{I}_p(\zeta(x_p))\nonumber\\&& +\lambda^{-1}\sum \limits_{0 < x_p< x}S_{r}(x-x_p)\hat{J}_p(\zeta(x_p)).
\end{eqnarray*}
Using the hypotheses, we have
\begin{eqnarray*}
	E\left\|\zeta(x)\right\|^2&\leq&8E\left\|\lambda^{-1}C_{r}(x)[\phi(0)+h(\zeta)]\right\|^2+8E\left\|\lambda^{-1}S_{r}(x)\big[\xi+f_1\big(0,\zeta(0),\zeta(\nu_1(0))\big)\big]\right\|^2 \\&&+8E\left\|\lambda^{-1}\displaystyle\int_{0}^{x}C_{r}(x-y)f_1\big(y,\zeta(y),\zeta(\nu_1(y))\big)dy\right\|^2\\&&+8E\left\|\lambda^{-1}\displaystyle\int_{0}^{x}P_{r}(x-y)\mathcal{B} u_\delta (y)dy\right\|^2\\&& +8E\left\|\lambda^{-1}\displaystyle\int_{0}^{x}P_{r}(x-y)f_2\big(y,\zeta(y),\zeta(\nu_2(y))\big)dy\right\|^2\\&&+8E\left\|\lambda^{-1}\displaystyle\int_{0}^{x}P_{r}(x-y)\int_{0}^{y} \varrho(y-\tau)g(\tau)dw(\tau)dy\right\|^2 \\&&+8E\Bigg\|\lambda^{-1}\sum \limits_{0<x_p<x}C_{r}(x-x_p)\hat{I}_p(\zeta(x_p))\Bigg\|^2\\&&+8E\Bigg\|\lambda^{-1}\sum \limits_{0<x_p<x}S_{r}(x-x_p)\hat{J}_p(\zeta(x_p))\Bigg\|^2\\	&\leq& 16 M_1\big[E\|\phi(0)\|^2+L_h(1+q)\big]+16M_2\big[E\|\xi\|^2+k_2\big]+8M_1bk_1\big(1+2q\big)\\&&+\frac{64}{\delta}M_3^2M_4^2b\bigg[2E\|\bar{\zeta}_b\|^2+2C_2\int_{0}^{b} E\left\|\varphi(y)\right\|^2_{L_2^0}dy+2M_1\big[E\|\phi(0)\|^2\\&&\quad+L_h(1+q)\big]+2M_2\big[E\|\xi\|^2+k_2\big]+M_1bk_1\big(1+2q\big)+M_3\int_0^b\Theta(y)\pounds(2q)dy\\&&\quad+C_2M_3\ell\beta(2q)\int_0^b\wp(y)dy+\sum \limits_{p=1}^nM_1L_{\hat{I}_p}+\sum \limits_{p=1}^nM_2L_{\hat{J}_p}\bigg]\\&&+8M_3\pounds(2q)\int_0^x\Theta(y)dy+8C_2M_3\ell\beta(2q)\displaystyle\int_{0}^{x}\wp(y) dy\\&&+8\sum \limits_{0<x_p<x}M_1L_{\hat{I}_p}+8\sum \limits_{0<x_p<x}M_2L_{\hat{J}_p}.\end{eqnarray*}
 $\implies \mathcal{Y}$ is bounded on $[0,b].$\\
\noindent Hence, it follows from Theorem~\ref{lm4} that $\mathcal{K}$ has a fixed point $\zeta \in \mathcal{PC}^1_b.$
\end{proof}

\begin{theorem}\label{th3} (Non-Convex Case)
	If the hypotheses (A1)-(A4), (A7-A8), (A9)(i), (A10)-(A11) are satisfied, then the system (\ref{1.1}) has at least one mild solution, provided that
	\begin{eqnarray}\label{eq5}
		\Lambda\bigg(1+\frac{6}{\delta}M_3^2M_4^2b\bigg)<1,
	\end{eqnarray}  where \begin{eqnarray*}
		\Lambda=14\bigg[M_1L_h+2L_{f_1}(M_2+M_1b)+2M_3bL_{f_2}+C_2\displaystyle \int_0^b M_3\ell\hat\wp(y)dy+\sum \limits_{p=1}^nM_1L_{\hat{I}_p}+\sum \limits_{p=1}^nM_2L_{\hat{J}_p}\bigg].\end{eqnarray*}
\end{theorem}
\begin{proof}
Since $S_{\mathcal{G},\zeta}\neq \emptyset$. Therefore, $\mathcal{G}$ has a nonempty measurable selection. We shall show that the operator $\mathcal{K}$ defined in Theorem \ref{th1} satisfies the assumption of Lemma~\ref{lm5}. The proof is as follows.\\
\textbf{Step 1.} $\mathcal{K}(\zeta)\in P_{cl}({\mathcal{PC}^1_b}) $ for each $\zeta \in {\mathcal{PC}^1_b}.$\\
Indeed, let $(\eta_n), n\geq 0 \in \mathcal{K}(\zeta)$ such that $\eta_n \to \eta.$ Then, $\eta \in {\mathcal{PC}^1_b}$ and there exists $g_n\in S_{\mathcal{G},\zeta_n}$ such that for each $x\in [0,b], \eta_n(x)$ is defined in (\ref{eq4}). Using (A11), we have for a.e. $x \in[0,b]$
$$\|g_n(x)\|^2\leq \hat \wp(x)+\hat\wp(x)\|\zeta\|^2.$$
The Lebesgue dominated convergence theorem implies that
$$\|g_n-g\|_{L_2^0}\to 0~\mbox{as}~ n \to \infty.$$
Hence, $g\in S_{\mathcal{G},\zeta}.$ Then, for each $x\in[0,b],~\eta_n(x) \to \eta(x),$ where 
\begin{eqnarray*}
	\eta(x)&=&C_{r}(x)[\phi(0)+h(\zeta_1)]+S_{r}(x)\big[\xi+f_1\big(0,\zeta_1(0),\zeta_1(\nu_1(0))\big)\big] \\&&-\displaystyle\int_{0}^{x}C_{r}(x-y)f_1\big(y,\zeta_1(y),\zeta_1(\nu_1(y))\big)dy+\displaystyle\int_{0}^{x}P_{r}(x-y)\mathcal{B} u_{\delta,1} (y)dy\\&& +\displaystyle\int_{0}^{x}P_{r}(x-y)f_2\big(y,\zeta_1(y),\zeta_1(\nu_2(y))\big)dy+\displaystyle\int_{0}^{x}P_{r}(x-y)\int_{0}^{y} \varrho(y-\tau)g(\tau)dw(\tau)dy \\&&+\sum \limits_{0<x_p<x}C_{r}(x-x_p)\hat{I}_p(\zeta_1(x_p))+\sum \limits_{0<x_p<x}S_{r}(x-x_p)\hat{J}_p(\zeta_1(x_p)).
\end{eqnarray*}
So, $\eta \in \mathcal{K}(\zeta).$\\
\textbf{Step 2.} There exists $\gamma<1$ such that $H_d\big(\mathcal{K}(\zeta_1), \mathcal{K}(\zeta_2)\big) \leq \gamma \|\zeta_1-\zeta_2\|_{\mathcal{PC}^1_b}$ for each $\zeta_1, \zeta_2 \in {\mathcal{PC}^1_b}.$ Let $\zeta_1, \zeta_2 \in {\mathcal{PC}^1_b}$ and $\eta \in \mathcal{K}(\zeta).$ Then, there exists $g_1\in S_{\mathcal{G},\zeta_1} $ such that $\eta(x)$ is defined as above. From (A11), it follows that
	$$H_d\big((\hat{\mathcal{G}}\zeta_1)(x),(\hat{\mathcal{G}}\zeta_2)(x)\big)\leq \hat\wp(x)\|\zeta_1(x)-\zeta_2(x)\|^2.$$
	Hence, there exists $g_2 \in S_{\mathcal{G},\zeta_2}$ such that $$\displaystyle\int_{0}^{x}\|g_1(x)-g_2(x)\|^2dx\leq \hat\wp(x)[\|\zeta_1(x)-\zeta_2(x)\|^2],\quad x\in[0,b].$$
	Consider the map $S:[0,b]\to P(\mathcal{H})$ defined by
	\begin{eqnarray*}S(x)=\bigg\{g_2(x)|g_2:[0,b] \to \mathcal{H} ~\mbox{is Lebesgue integrable and}~\displaystyle\int_{0}^{x}\|g_1(x)-g_2(x)\|^2dx\\\leq \hat\wp(x)\|\zeta_1(x)-\zeta_2(x)\|^2\bigg\}.
	\end{eqnarray*}
	Since the multivalued operator $S(x)\cap(\hat{\mathcal{G}}\zeta_2)(x) $ is measurable, there exists a function $\bar g_1(x)$ which is a measurable selection for $S.$ So, $\bar g_1(x) \in S_{\mathcal{G},\zeta_2},$ and for each $x\in[0,b],$
	 $$\displaystyle\int_{0}^{x}\|g_1(x)-\bar g_1(x)\|^2 dx\leq \hat\wp(x)\|\zeta_1(x)-\zeta_2(x)\|^2.$$
	 Let us define 
	 \begin{eqnarray*}
	 	\bar\eta(x)&=&C_{r}(x)[\phi(0)+h(\zeta_2)]+S_{r}(x)\big[\xi+f_1\big(0,\zeta_2(0),\zeta_2(\nu_1(0))\big)\big] \\&&-\displaystyle\int_{0}^{x}C_{r}(x-y)f_1\big(y,\zeta_2(y),\zeta_2(\nu_1(y))\big)dy+\displaystyle\int_{0}^{x}P_{r}(x-y)\mathcal{B} u_{\delta,2} (y)dy\\&& +\displaystyle\int_{0}^{x}P_{r}(x-y)f_2\big(y,\zeta_2(y),\zeta_2(\nu_2(y))\big)dy+\displaystyle\int_{0}^{x}P_{r}(x-y)\int_{0}^{y} \varrho(y-\tau)\bar g(\tau)dw(\tau)dy \\&&+\sum \limits_{0<x_p<x}C_{r}(x-x_p)\hat{I}_p(\zeta_2(x_p))+\sum \limits_{0<x_p<x}S_{r}(x-x_p)\hat{J}_p(\zeta_2(x_p)).
	 \end{eqnarray*}
 Then, for each $x\in[0,b]$, we get\\
 $ E\|\eta(x)-\bar\eta(x)\|^2$
  \begin{eqnarray*}
&\leq&14\|C_{r}(x)\|^2E\|h(\zeta_1)-h(\zeta_2)\|^2\\&&+14\|S_{r}(x)\|^2E\big\|f_1\big(0,\zeta_1(0),\zeta_1(\nu_1(0))\big)-f_1\big(0,\zeta_2(0),\zeta_2(\nu_1(0))\big)\big\|^2 \\&&+14\displaystyle\int_{0}^{x}\|C_{r}(x-y)\|^2E\big\|f_1\big(y,\zeta_1(y),\zeta_1(\nu_1(y))\big)-f_1\big(y,\zeta_2(y),\zeta_2(\nu_1(y))\big)\big\|^2dy\\&&+14\displaystyle\int_{0}^{x}\|P_{r}(x-y)\|^2\|\mathcal{B}\|^2E\| u_{\delta,1}-u_{\delta,2}\|^2 (y)dy\\&& +14\displaystyle\int_{0}^{x}\|P_{r}(x-y)\|^2E\big\|f_2\big(y,\zeta_1(y),\zeta_1(\nu_2(y))\big)-f_2\big(y,\zeta_2(y),\zeta_2(\nu_2(y))\big)\big\|^2dy\\&&+14C_2\displaystyle\int_{0}^{x}\|P_{r}(x-y)\|^2\int_{0}^{y} |\varrho(y-\tau)|^2\|g(\tau)-\bar g(\tau)\|^2_{L_2^0}d\tau dy \\&&+14\sum \limits_{0<x_p<x}\|C_{r}(x-x_p)\|^2\|\hat{I}_p(\zeta_1(x_p))-\hat{I}_p(\zeta_2(x_p))\|^2\\&&+14\sum \limits_{0<x_p<x}\|S_{r}(x-x_p)\|^2\|\hat{J}_p(\zeta_1(x_p))-\hat{J}_p(\zeta_2(x_p))\|^2\\
 \\&\leq & 14M_1L_h\|\zeta_1-\zeta_2\|^2+28M_2L_{f_1}\|\zeta_1-\zeta_2\|^2+28M_1bL_{f_1}\|\zeta_1-\zeta_2\|^2\\&&+14\displaystyle\int_{0}^{x}\|P_{r}(x-y)\|^2\|\mathcal{B}\|^2E\| u_{\delta,1}-u_{\delta,2}\|^2 (y)dy\\&&+28M_3bL_{f_2}\|\zeta_1-\zeta_2\|^2+14C_2\displaystyle\int_{0}^{x}M_3\ell \hat\wp(y)\|\zeta_1-\zeta_2\|^2 dy\\&&+14\sum \limits_{0<x_p<x}M_1L_{\hat{I}_p}\|\zeta_1-\zeta_2\|^2+14\sum \limits_{0<x_p<x}M_2L_{\hat{J}_p}\|\zeta_1-\zeta_2\|^2\\&\leq&
 14\bigg[M_1L_h+2L_{f_1}(M_2+M_1b)+2M_3bL_{f_2}+C_2\displaystyle \int_0^x M_3\ell\hat\wp(y)dy\\&&\quad+\sum \limits_{0<x_p<x}M_1L_{\hat{I}_p}+\sum \limits_{0<x_p<x}M_2L_{\hat{J}_p}\bigg]\|\zeta_1-\zeta_2\|^2\\&&+\frac{84}{\delta}M_3^2M_4^2b\bigg[M_1L_h+2L_{f_1}(M_2+M_1b)+2M_3bL_{f_2}+C_2\displaystyle \int_0^b M_3\ell\hat\wp(y)dy\\&&\quad+\sum \limits_{p=1}^nM_1L_{\hat{I}_p}+\sum \limits_{p=1}^nM_2L_{\hat{J}_p}\bigg]\|\zeta_1-\zeta_2\|^2\\
 &\leq& 2 \Lambda\bigg(1+\frac{6}{\delta}M_3^2M_4^2b\bigg)\|\zeta_1-\zeta_2\|^2,
 \end{eqnarray*}
where \begin{eqnarray*}
	\Lambda=14\bigg[M_1L_h+2L_{f_1}(M_2+M_1b)+2M_3bL_{f_2}+C_2\displaystyle \int_0^b M_3\ell\hat\wp(y)dy+\sum \limits_{p=1}^nM_1L_{\hat{I}_p}+\sum \limits_{p=1}^nM_2L_{\hat{J}_p}\bigg].\end{eqnarray*}
From the analogous relation obtained by interchanging the roles of $\zeta_1$ and $\zeta_2$, it follows that
\begin{eqnarray*}
	H_d\big(\mathcal{K}(\zeta_1),\mathcal{K}(\zeta_2)\big)&\leq&2 \Lambda\bigg(1+M_3^2M_4^2\frac{b^2}{\delta}\bigg)\|\zeta_1-\zeta_2\|^2.
\end{eqnarray*}
From (\ref{eq5}), $\mathcal{K}$ is a contraction and thus, by Lemma~\ref{lm5}, $\mathcal{K}$ has a fixed point which is the solution to (\ref{1.1}) on $[0,b]$.
			\end{proof}
\begin{theorem}\label{th2}
	If (A1)-(A9) hold, and $f_1, f_2$ and $g$ are uniformly bounded, then the system (\ref{1.1}) is approximately controllable on $[0,b]$.
\end{theorem}
\begin{proof}
	Let $\zeta^\delta(\cdot) \in \mathcal{Z}_q$ be a fixed point of the operator $\mathcal{K}$. From Theorem \ref{th1}, any fixed point of $\mathcal{K}$ is a mild solution of the system (\ref{1.1}). This means that there is $\zeta^\delta\in \mathcal{K}(\zeta^\delta) $ with $g^\delta \in S_{\mathcal{G}, \zeta^\delta}$ such that
	\begin{eqnarray*}
		\zeta^\delta(x)&=&\bar\zeta_b- \delta \big(\delta I+\mu_0^b\big)^{-1}\bigg(E\bar{\zeta}_b
		-C_{r}(b)[\phi(0)+h(\zeta^\delta)]-S_{r}(b)\big[\xi+f_1\big(0,\zeta^\delta(0),\zeta^\delta(\nu_1(0))\big)\big]\bigg)\\&&\quad+\int_{0}^{b}\delta \big(\delta I+\mu_y^b\big)^{-1} \varphi(y)dw(y)+\int_0^b\delta \big(\delta I+\mu_y^b\big)^{-1}C_{r}(b-y)f_1\big(y,\zeta^\delta(y),\zeta^\delta(\nu_1(y))\big)dy\\&&\quad-\int_0^b\delta \big(\delta I+\mu_y^b\big)^{-1}P_{r}(b-y)f_2\big(y,\zeta^\delta(y),\zeta^\delta(\nu_2(y))\big)dy\\&&\quad-\displaystyle\int_{0}^{b}\delta \big(\delta I+\mu_y^b\big)^{-1}P_{r}(b-y)\int_{0}^{y} \varrho(y-\tau)g^\delta(\tau)dw(\tau)dy\\&&\quad-\delta\big(\delta I+\mu_y^b\big)^{-1}\sum \limits_{p=1}^nC_{r}(b-x_p)\hat{I}_p(\zeta^\delta(x_p)) -\delta \big(\delta I+\mu_y^b\big)^{-1}\sum \limits_{p=1}^nS_{r}(b-x_p)\hat{J}_p(\zeta^\delta(x_p)).
	\end{eqnarray*}
In view of the unifrom boundedness of $f_1, f_2$ and $g$, there are subsequences still denoted by $f_1\big(y,\zeta^\delta(y),\zeta^\delta(\nu_1(y))\big),$ $f_2\big(y,\zeta^\delta(y),\zeta^\delta(\nu_2(y))\big)$ and $g^\delta(\tau)$, which converge weakly to $f_1(y), f_2(y)$ and $g(\tau).$ Moreover, the operator $\delta\big(\delta I+\mu_0^b\big)^{-1} \rightarrow 0$ strongly as $\delta\rightarrow 0^+$ and $\big\|\delta\big(\delta I+\mu_0^b\big)^{-1}\big\| \leq 1$. Hence, from the Lebesgue-dominated convergence theorem and the compactness of $C_r(x)$ and $S_r(x), x>0,$ yield
	\begin{eqnarray*}
	E\left\|\zeta^\delta(x)-\bar\zeta_b\right\|^2&\leq& 	10E\Big\|\delta \big(\delta I+\mu_0^b\big)^{-1}\Big(E\bar{\zeta}_b
	-C_{r}(b)[\phi(0)+h(\zeta^\delta)]\\&&\quad\quad-S_{r}(b)\big[\xi+f_1\big(0,\zeta^\delta(0),\zeta^\delta(\nu_1(0))\big)\big]\Big)\Big\|^2\\&&
	+10C_2E\bigg(\int_{0}^{b} \big\|\delta \big(\delta I+\mu_y^b\big)^{-1}\varphi(y)\big\|_{L_2^0}^2dy\bigg)\\&&
	+10E\bigg(\int_0^b \big\|\delta \big(\delta I+\mu_y^b\big)^{-1}C_{r}(b-y)\big[f_1\big(y,\zeta^\delta(y),\zeta^\delta(\nu_1(y))\big)-f_1(y)\big]\big\|dy\bigg)^2\\&&
	+10E\bigg(\int_0^b \big\|\delta \big(\delta I+\mu_y^b\big)^{-1}C_{r}(b-y)f_1(y)\big\|dy\bigg)^2\\&&+10E\bigg(\int_0^b\big\|\delta \big(\delta I+\mu_y^b\big)^{-1}P_{r}(b-y)\big[f_2\big(y,\zeta^\delta(y),\zeta^\delta(\nu_2(y))\big)-f_2(y)\big]\big\|dy\bigg)^2\\&&+10E\bigg(\int_0^b\big\|\delta \big(\delta I+\mu_y^b\big)^{-1}P_{r}(b-y)f_2(y)\big\|dy\bigg)^2\\&&+10C_2E\bigg(\displaystyle\int_{0}^{b}\Big\|\delta \big(\delta I+\mu_y^b\big)^{-1}P_{r}(b-y)\int_{0}^{y} \varrho(y-\tau)\big[g^\delta(\tau)-g(\tau)\big]\Big\|d\tau dy\bigg)^2\\&&+10C_2E\bigg(\displaystyle\int_{0}^{b}\Big\|\delta \big(\delta I+\mu_y^b\big)^{-1}P_{r}(b-y)\int_{0}^{y} \varrho(y-\tau)g(\tau)\Big\|d\tau dy\bigg)^2\\&&+10E\bigg(\Big\|\delta\big(\delta I+\mu_y^b\big)^{-1}\sum \limits_{p=1}^nC_{r}(b-x_p)\hat{I}_p(\zeta^\delta(x_p))\Big\|\bigg)^2 \\&&+10E\bigg(\Big\|\delta \big(\delta I+\mu_y^b\big)^{-1}\sum \limits_{p=1}^nS_{r}(b-x_p)\hat{J}_p(\zeta^\delta(x_p))\Big\|\bigg)^2 \to 0~\mbox{as}~ \delta \to 0^+.
\end{eqnarray*}
Thus, $\zeta^\delta(x) \to\bar\zeta_b$ as $\delta \to 0^+$ and this shows that the system (\ref{1.1}) is approximately controllable. 
\end{proof}
\begin{remark}
	The case when  $f_1\big(x,\zeta(x),\zeta(\nu_1(x))\big)=\sigma(x,\zeta_x),$ $f_2\big(x,\zeta(x),\zeta(\nu_2(x))\big)=F(x,\zeta_x),$ $\mathcal{G}\big(x,\zeta(x),\zeta(\nu_3(x))\big)={G}(x,\zeta_x)$ and $h(\zeta)(x)=q(\zeta_{x_1},\zeta_{x_2},\ldots,\zeta_{x_m})$ with integer-order derivative instead of Caputo fractional derivative has been discussed in \cite{SU2022}.
\end{remark}
\begin{remark}
	The case when $f_1\big(x,\zeta(x),\zeta(\nu_1(x))\big)=0,$ $\mathcal {A}\zeta(x)=\mathcal {A}^{C}{D}^{p}\zeta(x),~(0<p<1),$ $f_2\big(x,\zeta(x),\zeta(\nu_2(x))\big)=F(x,\zeta(x)),$ $\mathcal{G}\big(x,\zeta(x),\zeta(\nu_3(x))\big)={G}(x,\zeta(x))$ has been studied in finite dimensional space in \cite{SB2018} without nonlocal and impulsive conditions.
\end{remark}
\section{\textbf{Application}}
\textbf{Example.} Consider the following system
\begin{eqnarray}  \label{5.1}
	\left \{ \begin{array}{lll} \partial_x^{\frac{4}{3}}\big[\zeta(x,z)+\tilde f_1\big(x,\zeta(x,z),\zeta(x-1,z)\big)\big]\in \dfrac{\partial^2\zeta (x,z)}{\partial x^2}+\tilde f_2\big(x,\zeta(x,z),\zeta(x-1,z)\big) \\\hspace{0.5 cm}+\mathcal{B} u(x,z)+\displaystyle\int_{0}^{x}|x-s|\tilde{\mathcal{G}}\big(s,\zeta(s,z),\zeta(s-1,z))\big)dw(s),\quad
	 z \in [0,\pi],~  x\in [0,1],
		&\\\zeta(x,z)=\phi(x,z)+h(\zeta), ~ \dfrac{\partial \zeta(x,z)}{\partial x}\Big\rvert_{x=0}=2+z, \quad x\in [-1,0], ~z\in [0,\pi],
		&\\ \Delta \zeta(x,z)\Big\rvert_{x=x_p}=\displaystyle\int_{-1}^{x_p}  \beta_p(x_p-s)\zeta(s,z)ds,\quad p=1,2,\ldots,n,
			&\\ \Delta \dfrac{\partial \zeta(x,z)}{\partial x}\Big\rvert_{x=x_p}=\displaystyle\int_{-1}^{x_p} \hat \beta_p(x_p-s)\zeta(s,z)ds,\quad p=1,2,\ldots,n,
		&\\\zeta(x,0)=\zeta(x,\pi)=0, \quad x \in [0,1],
	\end{array}\right.
\end{eqnarray}
where $\partial_x^{\frac{4}{3}}$ represents fractional partial derivative of Caputo type and $0=x_0<x_1<x_2<\cdots<x_n<x_{n+1}=1$ are impulse points. $w(x)$ represents $\mathcal{W}$-valued standard Wiener process on a complete probability space $(\Omega, \Upsilon,\Upsilon_x, P)$ with usual conditions and $u$ is a control function. $\beta_p, \hat \beta_p$ are continuous and satisfy Lipschitz condition. However, $h$ can be chosen in such a way that assumption (A8) holds. \\
Let $\mathcal{H}=\mathcal{W}=\mathcal{L}^2[0,\pi].$  Let the operator 
$A: \mathcal{H}\rightarrow \mathcal{H}$
be defined by
$A\zeta(x)=\dfrac{\partial^2 \zeta(x)}{\partial x^2}$
with $$D(A)=\big\{\zeta \in \mathcal{H} : \zeta, \zeta'~ \mbox{are absolutely continuous},~ \zeta'' \in \mathcal{H}~~ \mbox{and}~\zeta(0)=\zeta(\pi)=0\big\}.$$ 
The operator $A$ has discrete spectrum with normalized eigenvectors $e_n(z)=\sqrt{\frac{2}{\pi}}\sin (nz) $ corresponding to the eigenvalues $\lambda_n=-n^2,$ where $n \in \mathbb{N}.$ Moreover, $\{e_n:n \in \mathbb{N}\}$ forms an orthonormal basis for $\mathcal{H}.$ Thus, we have $$A\zeta=\sum \limits_{n \in \mathbb{N}}{-n^2} \langle \zeta,e_n \rangle e_n, ~ \zeta\in D(A).$$ 
$A$ generates a strongly continuous cosine family given by
$$C(x)\zeta=\sum \limits_{n \in \mathbb{N}}\cos(nx)\langle \zeta,e_n \rangle e_n,$$
and that the associated sine family is given by
$$S(x)\zeta=\sum \limits_{n \in \mathbb{N}}\frac{1}{n}\sin(nx)\langle \zeta,e_n \rangle e_n.$$
\noindent
From the subordinate principle \cite[Theorem 3.1]{{B2001}}, it follows that $A$ generates a strongly continuous exponentially bounded fractional cosine family $C_r(x),$ such that $C_r(0)=I$ and
\begin{eqnarray}\label{c1}
	C_r(x)=\int_0^\infty \psi_{x,\frac{r}{2}}(s)C(s)ds, \quad x>0,
\end{eqnarray}
where $\psi_{x,\frac{r}{2}}(s)=x^{-\frac{r}{2}}\chi_{\frac{r}{2}}(sx^{-\frac{r}{2}})$ and 
 $$\chi_{\rho}(z)=\sum\limits_{n=0}^{\infty}\frac{(-z)^n}{n!\Gamma(-\rho n+1-\rho)},~~ 0<\rho<1.$$

\noindent Define 
$$\mathcal{U}=\bigg\{u:u=\sum\limits_{n=2}^{\infty}u_n e_n~\Big\rvert \sum\limits_{n=2}^{\infty}u_n ^{2}<\infty\bigg\},$$
with the norm
$\|u\|=\bigg(\sum\limits_{n=2}^{\infty}u_n ^{2}\bigg)^{1/2}.$\\
Define the operator $\mathcal{B}:\mathcal{U}\rightarrow \mathcal{H}$ by $\mathcal{B}u=\big(\mathcal{B}u\big)(x),$ where $\mathcal{B} \in \mathcal{L(U,H)}$ such that $\mathcal{B}u(x)=2u_2(x)e_1(z)+\sum\limits_{n=2}^{\infty}u_n(x)e_n(z).$ \\
Next, we verify the hypotheses for the system (\ref{5.1}) one by one.\\
\noindent\textbf{Verification of A1:}\\
From (\ref{c1}), $C_r(x)$ is compact and  undoubtedly, $\|C_r\|^2\leq 1.$ Moreover, there exists $M_4>0$ such that $\|\mathcal{B}\|^2\leq M_4.$\\
\noindent\textbf{Verification of A2:}\\
\noindent Assume that $f_1\big(x,\zeta(x),\zeta(\nu_1(x))\big)(z)=\tilde f_1\big(x,\zeta(x,z),\zeta(x-1,z)\big).$\\
Define the operator $\tilde f_1:[0,1]\times\mathcal{PC}\times \mathcal{PC}_{0} \to \mathcal{H},$ by\\
$\tilde f_1\big(x,\zeta(x,z),\zeta(x-1,z)\big)=\dfrac{2e^{-x}\sin x}{\sqrt{3}+|\zeta(x,z)|+|\zeta(x-1,z)|}$.\\
Therefore,
\begin{eqnarray*}
	&&E\|\tilde f_1\big(x,\zeta(x,z),\zeta(x-1,z)\big)-\tilde f_1\big(x,\tilde \zeta(x,z),\tilde\zeta(x-1,z)\big)\|^2\\&&\quad=E\bigg\|\frac{2e^{-x}\sin x}{\sqrt{3}+|\zeta(x,z)|+|\zeta(x-1,z)|}-\frac{2e^{-x}\sin x}{\sqrt{3}+|\tilde\zeta(x,z)|+|\tilde\zeta(x-1,z)|}\bigg\|^2\\&&\quad\le2E\Big(\|\zeta(x,z)-\tilde\zeta(x,z)\|^2+\|\zeta(x-1,z)-\tilde\zeta(x-1,z)\|^2\Big).
\end{eqnarray*}
\noindent\textbf{Verification of A3:}\\
Define $\tilde f_2\big(x,\zeta(x,z),\zeta(x-1,z)\big)=\dfrac{|\zeta(x,z)|+|\zeta(x-1,z)|}{5\pi+|\zeta(x,z)|+|\zeta(x-1,z)|}$.\\
Thus, there exists $L_{\tilde {f_2}}>0$ such that
\begin{eqnarray*}
	&&E\|\tilde f_2\big(x,\zeta(x,z),\zeta(x-1,z)\big)-\tilde f_2\big(x,\tilde \zeta(x,z),\tilde\zeta(x-1,z)\big)\|^2\\&& \quad\le L_{\tilde {f_2}}E\Big(\|\zeta(x,z)-\tilde\zeta(x,z)\|^2+\|\zeta(x-1,z)-\tilde\zeta(x-1,z)\|^2\Big).
\end{eqnarray*}
\noindent\textbf{Verification of A4:}\\
Define 
$\varrho:[0,1] \to \mathbb{R}$ be such that $\varrho(x-s)=|x-s|.$ Therefore, $\|\varrho(x-s)\|^2\le \ell,$ where $\ell>0$.
\noindent\textbf{Verification of A5:}\\
Define 
\begin{eqnarray}\label{q001}
\tilde{\mathcal{G}}\big(x,\zeta(x),\zeta(\nu_3(x))\big)(z)&=&\tilde{\mathcal{G}}\big(x,\zeta(x,z),\zeta(x-1,z)\big)\notag \\
&=&\big\{g\in \mathcal{H}: g_1\big(x,\zeta(x,z),\zeta(x-1,z)\big) \leq g \leq g_2\big(x,\zeta(x,z),\zeta(x-1,z)\big)\big\},\notag\\
\end{eqnarray}
where $g_1, g_2:[0,1] \times \mathcal{PC}\times \mathcal{PC}_{0} \rightarrow P_{bd,cl,cv}\big(L(\mathcal{W},\mathcal{H})\big).$ We assume that for each $x\in[0,1]$, $g_1$ is lower semi continuous and $g_2$ is upper semi continuous. \\
\noindent\textbf{Verification of A6:}\\
Assume that there exists $\wp\in L^1([0,1], \mathbb{R}^+)$ and $\Theta:\mathbb{R}^+ \to (0,\infty)$ is continuous increasing function such that
\begin{eqnarray}\label{q002}
	\max\bigg\{ \int_0^x E\|g_1\big(x,\zeta(x,z),\zeta(x-1,z))\big)\|^2,\int_0^x E\|g_2\big(x,\zeta(x,z),\zeta(x-1,z))\big)\|^2\bigg\}\notag \\\le \wp(x)\Theta\Big(E\|\zeta(x,z)\|^2+E\|\zeta(x-1,z)\|^2\Big).\hspace*{5.4cm}
\end{eqnarray}
\noindent\textbf{Verification of A7:}\\
From (\ref{5.1}),
$$\Delta \zeta(x,z)\Big\rvert_{x=x_p}=\displaystyle\int_{-1}^{x_p}  \beta_p(x_p-s)\zeta(s,z)ds,\quad p=1,2,\ldots,n$$
and $$\Delta \dfrac{\partial \zeta(x,z)}{\partial x}\Big\rvert_{x=x_p}=\displaystyle\int_{-1}^{x_p} \hat \beta_p(x_p-s)\zeta(s,z)ds,\quad p=1,2,\ldots,n,$$
and we consider the function $\hat{I}_p, \hat{J}_p:\mathcal{H}\rightarrow \mathcal{H}$ such that 
the required conditions are fulfilled.\\

If we define $\zeta(x)(z)=\zeta(x,z),$ $u(x)(z)=u(x,z),$ then system (\ref{5.1}) can be written as in the abstract form (\ref{1.1}). If all the hypotheses of Theorem~\ref{th1} are satisfied, then the system (\ref{5.1}) admits a mild solution.

\section{ \textbf{Conclusion}}
The existence and approximate controllability results for a class of nonlocal fractional neutral stochastic integrodifferential inclusions of order $1<\alpha<2$ with impulses are studied. Sufficient conditions for the existence of mild solutions are given for the cases when the multivalued map is convex and non-convex. Fractional calculus, cosine and sine families, stochastic analysis, and fixed point theorems are all applied to achieve the desired outcomes. Next, we derived sufficient conditions for the approximate controllability of the system. Finally, we illustrated the practicality of the finding using a theoretical scenario.

Second-order differential equations, for instance, can
be used to analyze the system of dynamical buckling of a hinged extensible beam \cite{S1950,WE1982}. The integrated process in continuous time, which can be made stationary, is best
described by second-order stochastic differential equations. For instance, engineers
can make use of second-order stochastic differential inclusions to describe mechanical
vibrations or the charge on a capacitor. On the other hand, in
control problems, we look for a control that steers the system from a given initial state to a prescribed trajectory or a desired final state. Moreover, while launching a rocket into
space, it may be desirable to have a precise path along with the preferred destination
for cost-effectiveness and collision avoidance. The efficiency of describing the real-life
phenomena by the fractional differential equation are more accurate than by classical differential
equations due to their nonlocal property. In the future, we plan to examine optimal controllability for non-instantaneous impulsive stochastic integrodifferential evolution inclusions of order $1<\alpha<2$.\\

\noindent
\textbf{Acknowledgements:} The first and third authors acknowledge UGC, India, for providing financial support through MANF F.82-27/2019 (SA-III)/ 4453 and F.82-27/2019 (SA-III)/191620066959, respectively.\\
\textbf{Competing Interest:} None.

\end{document}